\documentclass[11pt]{article}

\usepackage[utf8]{inputenc}
\usepackage[T1]{fontenc}
\usepackage[english]{babel}

\usepackage[a4paper, margin=1in]{geometry}

\usepackage{amsmath, amssymb, amsthm}
\usepackage{mathrsfs}
\usepackage{bm}          

\usepackage{graphicx}
\usepackage{float}
\usepackage{booktabs}
\usepackage{caption}
\usepackage{subcaption}

\newtheorem{theorem}{Theorem}[section]
\newtheorem{lemma}[theorem]{Lemma}

\newtheorem{corollary}[theorem]{Corollary}
\theoremstyle{remark}

\newcommand{\ba}{\begin{array}}
\newcommand{\ea}{\end{array}}

\newtheorem{example}{Example}

\newcommand{\trace}{\mathop{\mathrm{tr}}}
\DeclareMathOperator{\dive}{div}

\title{ Ergodic theorem for the differential equations with interaction}

\author{Andrey Dorogovtsev, Suli Liu, Kateryna Hlyniana }

\begin{document}

\maketitle

\begin{abstract}
For a differential equation with interaction, we investigate its ergodic properties. We apply the obtained results to study the limiting behavior of braid invariants associated with the flow of solutions. 
\end{abstract}
\bigskip



\pagestyle{myheadings}
\thispagestyle{plain}
\markboth{}{}

\section{Introduction}
In this paper, we investigate differential equations with interaction. This type of equation was originally introduced in the stochastic setting by Andrey A. Dorogovtsev  \cite{Andrey-1-2003}-\cite{Andrey-book}.

We consider the deterministic version of such equations, which take the following form:
\begin{equation}\label{eq-interaction}
\begin{cases}
{\rm d}x(u,t)=a(x(u,t),\mu_t){\rm d}t,\\	
x(u,0)=u,\,\,\,\, u\in \mathbb{R}^2,\\
\mu_t=\mu_0\circ x(\cdot,t)^{-1}.
\end{cases}
\end{equation}
Here, the measure $\mu_0$ on $\mathbb{R}^2$ represents the initial distribution of mass of particles in the medium.
 The coefficient $a(x,\mu)$ depends on the spatial variable $x$ and the measure $\mu,$ which describes the mass of particles in space.
 The trajectory of a particle starting at point $u\in \mathbb{R}^2$ is given by the funcion $\{x(u,t),\ t\geq0\}.$
The measure $\mu_t$ describes the distribution of the mass of particles at time $t$.
It has been shown that if the coefficient of the equation \eqref{eq-interaction} satisfies the Lipschitz condition with respect to spatial and the measure-valued variable,
 then there exists a unique solution to \eqref{eq-interaction} (see P78 in \cite{Andrey-book}).
Moreover, for each fixed $t\geq 0$ the mapping $x(\cdot, t): \mathbb{R}^2\rightarrow \mathbb{R}^2$ is a homeomorphism (see P87 in \cite{Andrey-book}).

The aim of our investigation is to describe the limiting behavior for invariants of braids driven by the flow of solutions. More precisely, let us fix $n$ points $u_1, \ldots, u_n \in \mathbb R^2$ and consider the corresponding trajectories:
$$\{x(u_1,t),\ldots, x(u_n,t)\}_{t\geq 0}.$$
These trajectories form  an $n-$braids by $n$ strings in the space $\mathbb R^2\times [0, T],$ connecting two parallel planes  $\mathbb R^2\times {0}$ and $\mathbb R^2\times {T}.$ For convenience, we represent the horizontal plane as the complex plane $\mathbb C,$  Thus, the trajectory of a particle  is described by function
$$z(u,t)= x_1(u,t)+ix_2(t,t) \in \mathbb C.$$
Under general assumptions of the equation with interaction, the strings cannot cross through each other. Consequently, a geometrical braid on $n$ trajectories can be represented as a single curve:
$$
b(t) = (z(u_1,t),\ldots, z(u_2,t))
$$
in the configuration space $\{\mathbb C^n: z_i\neq z_j \text{ for } i\neq j\}.$
These braids represent topological properties of the motion and have been studied in various contexts (see, for example, \cite{Arnold1978},  \cite{Moore1993},\cite{Clausen1998}).

The simplest invariant of such a braid is the winding number, which counts how many times two strings in a braid wrap around each other.
The winding number can be defined via the Gauss integral:
$$
\varphi(u_1, u_2, T) = \frac{1}{2\pi i}\int_0^T\frac{dz(u_1,t) - dz(u_2,t)}{z(u_1,t)-z(u_2,t)}
$$

V.A. Vassiliev introduced higher-order invariants for knots and braids \cite{Vassiliev1990}. His definition of finite-type (or Vassiliev) invariants is based on the observation that the space of knots forms a topological space. The knot invariants can be thought of as the locally constant functions on this space (see more details, for example, in \cite{Chmutov2012}). The system of Vassiliev's invariant is known to be complete in the sense that two braids are homotopically equivalent if and only if all their Vassiliev's invariants coincide \cite{BarNatan1995}.  M. Kontsevich showed that Vassiliev invariants can be expressed via integral formulas \cite{Kontsevich1995}, providing a powerful copmutational tool \cite{Berger2001}. 
The so-called Kontsevich integrals are constructed using winding numbers between each pair of strings in a braid. Define the following one-form:
$$
\omega_{i,j}(t) =  \frac{1}{2\pi i}\frac{dz(u_i,t) - dz(u_j,t)}{z(u_i,t)-z(u_j,t)}.
$$
Then the Vassiliev invariants of order $m$ can be represented as a linear combination of the integrals of the form:
$$
I_m(\vec u, t)=\int_{\Delta_m} \omega_{i_1, j_1}(t_1)\wedge \ldots \wedge \omega_{i_m, j_m}(t_m),
$$
where $\vec u \in \mathbb R^{2\times m},$ and 
$$\Delta_m = \{(t_1, \ldots, t_m): \ 0\leq t_1\leq \ldots \leq t_m\leq t\}$$
is the $m-$ dimensional simplex. Here, we use the standard notation for wedge product of differential forms. For instance, if $\beta(t)=\sum_{j=1}^k b_j(t)dt_{j_1}\wedge dt_{j_k}$  then the antisymmetry relation
$
dt_j\wedge dt_k = - dt_k\wedge dt_j
$
holds.

The goal of this paper is to analyze the asymptotic behavior of such integrals in the case where braids are generated by solutions to a differential equation with interaction. Specifically, we study the asymptotic behavior of Kontsevich integrals averaged with respect to the initial mass distribution. More precisely, we will find out conditions on the coefficient $a$ under which the limit
$$
\lim_{t\to\infty} C_m(t) \int_{\mathbb R^{2\times m}} I_m(\vec u, t) \mu_0^m(d\vec u).
$$
exists, where $C_m$ is an appropriate normalization factor and $\mu_0^{\otimes m}$ denotes the $m$-fold product of the initial distribution $\mu_0$.

We should note that the differential equation with interaction can be regarded as a model for $N-$vortex system \cite{Arnold-book}. More precisely, consider $N$ vortices with circulations, $\Gamma_i,$ $i=1,\ldots, N,$  located in the plane $\mathbb R^2.$
In this case, the vortices at any time are concentrated at $N$ points.  We denote  the position of a vortex at time $t$ by
\[
V(u,v;t):=\begin{bmatrix} x(u, t) \\ y(v, t) \end{bmatrix}, \quad \begin{bmatrix} x(u, 0) \\ y(v, 0) \end{bmatrix}=\begin{bmatrix} u \\ v \end{bmatrix}
\]
with $(u,v)\in \mathbb R^2.$  It is convenient to describe the vortex dynamics as a system in the configuration space $\mathbb{R}^{2N}$.
The evolution of vortices is governed by the Hamiltonian system:
  $$
  \renewcommand{\arraystretch}{1.5}
\left\{\begin{array}{l}
d x\left(u_{k}, t\right)=\Gamma_{k} \frac{\partial H}{\partial v_{k}}\left(x\left(u_{1}, t\right), \ldots, x\left(u_{N}, t\right), y\left(v_{1}, t\right), \ldots, y\left(v_{N}, t\right)\right) d t \\
d y\left(v_{k}, t\right)=-\Gamma_{k} \frac{\partial H}{\partial u_{k}}\left(x\left(u_{1}, t\right), \ldots, x\left(u_{N}, t\right), y\left(v_{1}, t\right), \ldots, y\left(v_{N}, t\right)\right) d t \\
k=1, \ldots, N
\end{array}\right.
$$
where Hamiltonian is defined as
$$
H\left(u_{1},\ldots, u_{k}, v_{1}, \ldots, v_{k}\right)=-\frac{1}{2 \pi} \sum_{i \neq j} \frac{1}{\Gamma_{i} \Gamma_{j}} \ln \sqrt{\left(u_{i}-u_{j}\right)^{2}+\left(v_{i}-v_{j}\right)^{2}}
$$
Calculating partial derivatives of $H,$ we obtain the right-hand side of the Hamiltonian equations as:

$$
\begin{aligned}
\Gamma_{k} \frac{\partial H}{\partial v_{k}}(x(\vec{u}, t), y(\vec{v}, t)) & =-\frac{1}{\pi} \Gamma_{k} \sum_{j \neq k} \frac{1}{\Gamma_{k} \Gamma_{j}} \frac{y\left(v_{k}, t\right)-y\left(v_{j}, t\right)}{\left(x\left(u_{k}, t\right)-x\left(u_{j}, t\right)\right)^{2}+\left(y\left(v_{k}, t\right)-y\left(v_{j}, t\right)\right)^{2}}= \\
& =-\frac{1}{\pi} \sum_{j \neq k} \frac{1}{\Gamma_{j}} \frac{y\left(v_{k}, t\right)-y\left(v_{j}, t\right)}{\left(x\left(u_{k}, t\right)-x\left(u_{j}, t\right)\right)^{2}+\left(y\left(v_{k}, t\right)-y\left(v_{j}, t\right)\right)^{2}}= \\
& =-\frac{1}{\pi} \int \varphi\left(x\left(u_{k}, t\right), y\left(v_{k}, t\right), u_{0}, v_{0}\right) \mu_{t}\left(d u_{0}, d v_{0}\right)
\end{aligned}
$$
where $\varphi\left(u, v, u_{0}, v_{0}\right)=\frac{v-v_{0}}{\left(u-u_{0}\right)^{2}+\left(v-v_{0}\right)^{2}} \mathbb{I}_{\left\{v \neq v_{0}\right\}}$
and $\mu_t$ is a push-forward of a discrete measure $\mu_0$ given by
$$
\mu_{0}=\sum_{j=1}^{N} \frac{1}{\Gamma_{j}} \delta_{\left(u_{j}, v_{j}\right)} ; \quad \mu_{t}=\mu_{0} \cdot V(\cdot, t)^{-1}=\sum_{j=1}^{N} \frac{1}{\Gamma_{j}} \delta_{V\left(u_{j}, v_{j}; t\right)}
$$
Similarly,

$$
\Gamma_{k} \frac{\partial H}{\partial u_{k}}=-\frac{1}{\pi} \int \psi\left(x\left(u_{k}, t\right), y\left(v_{k}, t\right), u_{0}, v_{0}\right) \mu_{t}\left(d u_{0}, d v_{0}\right)
$$
where $\psi\left(u, v, u_{0}, v_{0}\right)=\frac{u-u_{0}}{\left(u-u_{0}\right)^{2}+\left(v-v_{0}\right)^{2}} \mathbb{I}_{\left\{u{\neq u_{0}}\right\}}.$
Hence, the Hamiltonian system can be rewritten as a differential equation with interaction:
$$
\left\{\begin{array}{l}
d x(u, t)=-\frac{1}{\pi} \int_{\mathbb{R}^{2}} \varphi\left(x(u, t), y(v, t), u_{0}, v_{0}\right) \mu_{t}\left(d u_{0} d v_{0}\right) d t \\
d y(v, t)=\frac{1}{\pi} \int_{\mathbb{R}^{2}} \psi\left(x(u, t), y(v, t), u_{0}, v_{0}\right) \mu_{t}\left(d u_{0} d v_{0}\right) d t, \\
x(u, 0)=u, \quad u\in \mathbb R,  \\
y(v, v)=v, \quad u\in \mathbb R,  \\
\mu_{t}=\mu_{0} \cdot(x(\cdot, t), y(\cdot, t))^{-1},
\end{array}\right.
$$
with $\mu_0= \sum_{j=1}^{N} \frac{1}{\Gamma_{j}} \delta_{\left(u_{j}, v_{j}\right)}.$
In general, the circulations  $\Gamma_k$ can be negative, in which case $\mu_0$ becomes a signed measure rather than a probability measure. Importantly, this formulation allows us to extend the dynamics beyond the $N$ vortices to all particles in the flow, for all initial points $(u, v) \in \mathbb{R}^2$. 

In the special case of discrete measure $\mu_{0}=\sum_{j=1}^{N} \frac{1}{\Gamma_{j}} \delta\left(u_{j}, v_{j}\right)$
the heavy points 
$$\begin{bmatrix} x(u_i, t) \\ y(v_i, t) \end{bmatrix},
\quad i= 1,\ldots, N
$$ 
represent the $N$ vortices. The Hamiltonian system determines the dynamics of these vortices.
Once their motion is known, we can describe the behavior of other particles with negligible mass—interpreted as `dust particles'—in the flow generated by the vortices.
For such `zero-mass' particles, the equation becomes:
$$
\left\{\begin{array}{l}
d x(u, t)=-\frac{1}{\pi} \sum_{j=1}^{N} \frac{y(v, t)-y\left(v_{j}, t\right)}{\left( x(u, t)-x\left(u_{j}, t\right)\right)^{2}+\left(y(v, t)-y\left(v_{j}, t\right)\right)^2}d t \\
d y(v, t)=\frac{1}{\pi} \sum_{j=1}^{N} \frac{x(u, t)-x(u_j, t)}{\left(x(u, t)-x\left(u_{j}, t\right)\right)^{2}+\left(y(v, t)-y\left(v_{j,}, t\right)^{2}\right.} d t \\
(u, v) \in \mathbb{R}^{2} \backslash\left\{\left(u_{i}, v_{i}\right)\right\}_{j=1}^{N}
\end{array}\right.
$$
This system describes the motion of dust particles in the presence of $N$ vortices. Such interpretation of `zero-mass' particles was considered in the work by  V. Yu. Belashov \cite{belashov2017interaction}, where the author demonstrated $N-$vortex simulation and behaviour of dust particles.

We can generalize the model by considering not only discrete initial measures but also absolutely continuous ones with respect to the Lebesgue measure. Informally, this corresponds to a flow consisting of an infinite number of infinitesimal vortices continuously distributed in space, each contributing to and influenced by the global vorticity field. In this setting, the point-vortex dynamics transitions into a continuum model that describes the evolution of a smooth vorticity distribution. Such a case is described by a general equation with interaction.

There is a deep connection between the invariants of a vector field and braids formed by trajectories induced by this field.  In the study of force-free magnetic fields, L. Woltjer and J-J. Moreau, independently discovered that such fields possess a conserved quantity known as helicity \cite{Woltjer1958, Moreau1961}. 
Let $\mathbb S^3$ denote the 3-dimensional sphere equipped with a volume form $\mu$ and let $\mathbf V$ be a vector field that preserves the volume $\mu-$.  The \textit{helicity} of a vector field $\mathbf{V}= \mathbf{V}(\mathbf{x},t)$ is defined by the integral
$$
\mathcal{H}= \int_{\mathbb S^3} \mathbf{V}\cdot (\nabla \times \mathbf{V}) d\mu.
$$
 H.K. Moffatt demonstrated that helicity is invariant under volume-preserving diffeomorphisms of a flow. He also pointed out that helicity is related to knot theory and, more precisely, to the linking number \cite{Moffatt1969}. This connection was further investigated by V. Arnold, who interpreted helicity as the average linking number of the field's trajectories \cite{Arnold1973}.
To make this precise, let $lk_{\mathbf{V}} (x_1,x_2; T_1, T_2):= lk(\Gamma_1, \Gamma_2)$ denote the linking number of the oriented curves $\Gamma_1$, $\Gamma_2$ formed by tracing the trajectories of $\mathbf V$ starting $x_1, \ x_2$ for time intervals $[0, T_1]$ and $[0, T_2]$ respectively. The \textit{asymptotic liking number} between these two  trajectories is then defined as   
$$
\lambda(x_1, x_2):=\lim_{T_1, T_2\to \infty} \frac{lk_{\mathbf{V} }(x_1,x_2; T_1, T_2)}{T_1 T_2},
$$
provided the limit exists.
V. I. Arnold proved the existence of this limit $(\mu\times \nu)$ for divergence-free vector fields and smooth measures \cite{Arnold1973}. This was later generalized by G. Cortreras and R. Iturriaga to arbitrary Borel probability measure on compact manifolds \cite{ContrerasIturriaga1999}.
Further, T. Vogel showed that this limit exists as an element of the space $L_1(\mathbb S^3\times \mathbb S^3)$ of Lebesgue integrable functions and is independent of the system of geodesics for any divergence-free field $\mathbf{V}$ \cite{Vogel2003}.

Moreover, E. Kudryavtseva \cite{Kudryavtseva2016}, and independently A. Enciso, D. Peralta-Salas, and F. Torres de Lizaur \cite{Enciso2016}, proved that any regular integral invariant of volume-preserving transformations is equivalent to helicity. That is, helicity is the only integral invariant of incompressible flows whose derivative is continuous in the $C^1$ topology.

In \cite{BaaderMarche2012}, the authors analyze the asymptotic growth of Vassiliev invariants on the flow lines of ergodic vector fields. They show that the helicity of a vector field completely determines the asymptotic behavior of its Vassiliev invariants. This establishes a strong connection between finite-type invariants and classical topological invariants arising in fluid dynamics.

The reverse connection—from knot theory to dynamical systems—was developed in \cite{SongBanksDiaz2006}, where the authors proposed an explicit construction of dynamical systems, defined via differential equations, that realize any given knot type. This demonstrates that not only do vector fields generate braids and knots, but given a knot, one can construct a flow whose trajectories reproduce it.

Since the origin of differential equations with interaction comes from the stochastic case, we also mention some related results in this context.
The asymptotic behaviour of the  angle $\phi(t), $ describing how a two-dimensional Brownian motion started from a non-zero point winds around the origin up
to time $t,$ was studied by F. Spitzer \cite{Spitzer1958}, Z. Shi \cite{Shi1994}, J. Bertoin and W. Werner \cite{BertoinWerner1994}.
V.A. Kuznetsov \cite{Kuznetsov2015b} investigated the possibility of obtaining the large-deviation principle for the winding angle of the Brownian motion.
The asymptotic behavior of mutual winding angles of several two-dimensional Brownian motions has been analyzed in various settings: for independent Brownian motions by  M. A. Berger and P. H. Roberts \cite{BergerRoberts1988}, M Yor \cite{Yor1991};
 for polarized Brownian stochastic flows by C. L. Zirbel and W. A. Woyczyński \cite{ZirbelWoyczynski2002}; and for isotropic Brownian stochastic flows by V. A. Kuzntesov \cite{Kuznetsov2017}.

It is important to note that each trajectory driven by a differential equation with interaction depends on the configuration of all other particles in the flow. In such settings, it is natural to expect that invariants of braids formed by these trajectories should be described in terms of the measure that represents the collective position of particles.
 This is the approach we adopt in the present work. Specifically, we study properties of measure-valued solutions to differential equations with interaction and show that braid invariants can be expressed in terms of these solutions. Furthermore, we use the ergodic theorem to analyze the asymptotic behavior of these braid invariants. Since the functional of interest depends on a measure-valued variable, verifying the conditions under which the ergodic theorem applies becomes a non-trivial task. 

 The paper is organized as follows. In Section 2, we discuss properties of the solution to the differential equation with interaction. We will concentrate on the evolutionary property of measure-valued solutions. Section 3 is devoted to an analog of the classical Liouville theorem for differential equations with interaction. In this section, we will prove that under conditions closely resembling the classical case, the solution $x(\cdot, t)$ preserves volume in $\mathbb R^d.$ In Section 4, we extend the Liouville-type result to the measure-valued setting. We identify a class of measures on the space of probability measures that remain invariant under the flow generated by measure-valued solutions.

In Section 5, we study the limiting behavior of braid invariants associated with the flow.

Section 6 is devoted to the ergodic theorem for measure-valued solutions to equations with interactions. Here, we write conditions on the coefficient of the equation under which one can apply the ergodic theorem. 

Each section of the paper contains examples demonstrating how the results can be applied to specific cases of equations. 

An alternative approach to establishing the existence of invariant measures is presented in Section 5. We describe conditions under which a measure-valued solution remains supported within a bounded closed domain in $\mathbb{R}^d$.

 \section{Evolution property for flows with interaction}
In this section, we introduce a semi-group generated by the solutions to the equation with interaction \eqref{eq-interaction}. In the introduction, we mentioned the known result about the existence and uniqueness of the solution when the coefficient $a$ satisfies the Lipschitz condition with respect to both variables.   For this reason, we restrict ourselves to this case. It is important to note that, generally, solutions to equation \eqref{eq-interaction} do not satisfy the evolutionary (semigroup) property. That is, in general, for a fixed point $u\in \mathbb R^2$
$$x_{0, \mu_s}(x_{0,\mu_0}(u,s),t)\neq x_{0,\mu_0}(u, t+s),$$
where $x_{0,\mu_0}$ denotes a solution to \eqref{eq-interaction} starting from time $0$ with initial mass distribution $\mu_0$.
The reason is that the motion of one point depends on the mass distribution of the system.
However, the process of mass distribution $\mu_t$ itself  evolves deterministically based on the initial measure $\mu_0$ and the flow map $x(\cdot, t),$ via
$$\mu_t = \mu_0\circ x(\cdot, t)^{-1}.$$
Thus, instead of tracking the  family $\{x_{0,\mu_0}(u, t), \ t\geq 0\},$ we focus on the evolution of the measure $\mu_t.$

Let  $\mathfrak{M}_0$ be the space of all probability measures on $\mathbb{R}^2$ with Wasserstein distance $\gamma_0$,
then $(\mathfrak{M}_0, \gamma_0)$ is a complete separable metric space
(P420, \cite{Dudley}).
Define a map $\Phi_t: \mathfrak{M}_0\rightarrow \mathfrak{M}_0$ as follows. For $\mu_0 \in \mathfrak{M}_0$, let $\mu_t$, $t\geq 0$, be a measure-valued function determined by equation~\eqref{eq-interaction}. Set
$$\Phi_t(\mu_0):=\mu_t.$$

\begin{lemma}
    The map $\Phi_t$, $t\geq 0$, satisfies the evolutionary property
\[
  \Phi_{t+s}=\Phi_{t}\circ\Phi_s\quad t,s \geq 0.
\]
\end{lemma}
\begin{proof}
   Indeed, for $\mu_0\in\mathfrak{M}_0$ consider the equation
\begin{equation*}
\begin{cases}
{\rm d}x_{s,\mu_s}(u,t+s)=a(x_{s,\mu_s}(u,t+s),\mu_{t+s}){\rm d}t,\\	
x_{s,\mu_s}(u,0)=u,\ u\in \mathbb{R}^d,\\
\mu_{t+s}=\mu_{s}\circ x_{s,\mu_s}(\cdot,t+s)^{-1}.
\end{cases}
\end{equation*}
Then
\begin{align*}
  \Phi_{t}\circ\Phi_s(\mu_0) & =\Phi_{t}(\mu_s)=\mu_s\circ x_{s, \mu_s}(\cdot, t+s)^{-1}=\mu_{t+s}=\Phi_{t+s}(\mu_0),
\end{align*}
so, $\Phi_{t+s}=\Phi_{t}\circ\Phi_s$ which implies that $\{\Phi_t\}_{t\geq 0}$ is a semi-group of mappings on $\mathfrak{M}_0$.
 
\end{proof}

Let us show that the $\Phi_t$ is a homeomorphism as a map on the space $\mathfrak{M}_0$.
We will prove the existence of inverse map $\Phi^{-1}.$ The idea is similar to the case of ordinary differential equations. Indeed,  we consider the following equations with interaction:

$$
\left\{\begin{array}{l}
d x(u, t)=a\left(x(u, t), \mu_{t}\right) d t \\
x(u, 0)=u, u \in \mathbb{R}^{d} \\
\mu_{t}=\mu_{0} \circ x(\cdot, t)^{-1}
\end{array}\right.
$$
and

$$
\left\{\begin{array}{l}
d \widetilde{x}(u, t)=-a\left(\widetilde{x}(v, t), \nu_{t}\right) d t \\
\tilde{x}(u, 0)=u, u \in \mathbb{R}^d \\
\nu_{t}=\nu_{0} \circ \tilde{x}(\cdot,  t)^{-1} .
\end{array}\right.
$$

The measure-valued solutions $\mu_{t}, \nu_{t}$ we denote as $\mu_{t}=\Phi_{t}\left(\mu_{0}\right), \quad \nu_{t}=\Psi_{t}\left(\nu_{0}\right).$
To demonstrate that $\Phi^{-1} = \Psi,$ we consider the following example of a linear equation with interaction, whose solution we can easily find.

\begin{example}
    
Let us consider a linear equation of the form

$$
\left\{\begin{array}{l}
d x(u, t)=\int_{\mathbb{R}^{d}} A(x(u, t)-v) \mu_{t}(d v) d t \\
x(u, 0)=u, \quad u \in \mathbb{R}^{d} \\
\mu_{t}=\mu \circ x(\cdot, t)^{-1}
\end{array}\right.
$$
Then the solution is given by
$$
x(u, t)=e^{A t} u+\left(I-e^{A t}\right) m,
$$
with $m=\int_{\mathbb{R}^{d}} v \mu(d v).$
The solution for the backward equation
$$
\left\{\begin{array}{l}
d y(u, t)=\int_{\mathbb{R}^{d}}(-A)(y(u, t)-v) \mu_{t}(d v) d t \\
y(u, T)=u, u \in \mathbb{R}^{d} \\
\mu_{t}=\mu \circ y(\cdot, t)^{-1}
\end{array}\right.
$$
is given by
$$
y(u, t)=e^{-A t} u+\left(I-e^{-A t}\right) m .
$$
Now,
$$
\begin{aligned}
& x(y(u, t), t)=e^{A t}\left(e^{-A t} u+\left(I-e^{-A t}\right) m\right)  +\left(I-e^{A t}\right) m= \\
= & u+\left(e^{A t}-I\right) m+\left(I-e^{A t}\right) m=u .
\end{aligned}
$$
Analogously, $y(x(u, t), t)=u$.

For measures we have
$$
\Phi_{t}(\mu)=\mu \circ x(\cdot, t)^{-1}, \quad \Psi_{t}(\mu)=\mu \circ y(\cdot, t)^{-1} .
$$
So, we have
$$
\Phi_{t}\left(\Psi_{t}(\mu)\right)=\left(\mu \circ y(\cdot, t)^{-1}\right) \circ x(
\cdot, t)^{-1}
$$
and  by the definition this means that for any subset $\Delta \in \mathbb{B}\left(\mathbb{R}^{d}\right)$,
$$
\begin{aligned}
& \Phi_{t}\left(\Psi_{t}(\mu)\right)(\Delta)=\left(\mu \circ y(\cdot, t)^{-1}\right)\{u: x(u, t) \in \Delta\}= \\
& =\mu\left\{v: y(v, t) \in x^{-1}(\Delta, t)\right\}=\mu \left\{ v: v \in y^{-1}\left(\cdot x^{-1}(\Delta, t)\right)\right\}= \\
& =\mu(t) .
\end{aligned}
$$
Analogously, $\Psi_{t}\left(\Phi_{t}(\mu)\right)=\mu$.
\end{example} 

The following lemma holds for a general case.

\begin{lemma}\label{phi-t-homeomorphism}
Let the coefficient $a$ be a Lipschitz function with respect to both variables. Then for any $\mu_{0}, t \geq 0$
$$
\Psi_{t}\left(\Phi_{t}\left(\mu_{0}\right)\right)=\mu_{0}, \quad \Phi_{t}\left(\Psi_{t}\left(\mu_{0}\right)\right)=\mu_{0}
$$   
\end{lemma}

\begin{proof}
We will start considerations from the following family of discrete measures. 
    For fixed $\left\{p_{k}\right\}_{k=1}^{n},\ p_{k}>0,\ \sum_{k=1}^{N} p_{k}=1$, consider a family of measures

$$
\mathfrak M_0^{\vec p}=\left\{\mu_{0}=\sum_{k=1}^{N} p_{k} \delta_{u_{k}}, u_{k} \in \mathbb{R}^{d}, \ u_i\neq u_j, \ i\neq j \right\}.
$$
For initial measure $\mu=\sum_{k=1}^{N} p_{k} \delta_{u_{k}},$ the mass distribution at time $t$ is equal to $\mu_t=\sum_{k=1}^{N} p_{k} \delta_{x(u_{k},t)},$ so 
$$
\Phi_t\left(\mathfrak M_0^{\vec p}\right)\subset \mathfrak M_0^{\vec p}.
$$
From the form of $\mu_t$ we can see that the mapping $\Phi_t$ depends only on the position of heavy particles $x(u_k,t),$ $k=1,\ldots, n.$ The equation with interaction for heavy parties $y_{k}(t):=x\left(u_{k}, t\right)$ can be considered as a system of ordinary differential equations. Indeed, let us denote $\quad \alpha_{k}\left(y_{1}, \ldots, y_{k}\right)=a\left(y_{k}, \sum_{j=1}^{N} p_{j} \delta_{y_{j}}\right)$. Then

$$
\left\{\begin{array}{l}
d y_{k}(t)=\alpha_{k}\left(y_{1}(t), \ldots, y_{k}(t)\right) d t, \\
y_{k}(0)=u_{k}, \ k=1, \ldots, N_{0}.
\end{array}\right. 
$$
Similarly, $\quad z_{k}(t):=\tilde{x}\left(u_{k}, t\right) \quad$ satisfy following  system of differential equations

$$
\left\{\begin{array}{l}
d z_{k}(t)=-\alpha_{k}\left(z_{1}(t), \ldots, z_{k}(t)\right) d t, \\
z_{k}(0)=u_{k},\ k=1, \ldots, N_{0}.
\end{array}\right. 
$$
Note that $\alpha_k,$ $k=1,\ldots, n$ are Lipschitz functions.

Let us denote the solutions to these Cauchy problems by
$$
\varphi_{t}(\vec{u})=\left(y_{1}(t), \ldots, y_{k}(t)\right), \quad \psi_{t}(\vec{u})=\left(z_{1}(t), \ldots, z_{k}(t)\right).
$$
It is known that in this case
$$
\varphi_{t}\left(\psi_{t}(\bar{u})\right)=\vec{u} \text { and } \psi_{t}\left(\varphi_{t}(\bar{u})\right)=\vec{u}.
$$
To the positions of the heavy particles $\phi_t(\vec u)$ we put into the correspondence their mass distribution. More precisely, we define mapping
$$
\mathcal M_{\vec p}: \mathbb{R}^{N} \rightarrow \mathfrak M_0^{\vec p}$$
by the formula 
$$\mathcal M_{\vec p}\left(u_{1},\ldots,  u_{k}\right)=\sum_{k=1}^{N} p_{k} \delta_{u_{k}}.
$$
We also define inverse mapping 
$$
\mathcal M^{-1}_{\vec p}: \mathfrak M_0^{\vec p}\rightarrow  \mathbb{R}^{N}$$
by the formula: for $\mu=\sum_{k=1}^{N} p_{k} \delta_{u_{k}}$
$$\mathcal M^{-1}_{\vec p}\left(\mu\right)=(u_{1},\ldots,  u_{k}),
$$
where in the right-hand side the points with equal masses are ordered in the alphabetical order (denoted by $\preceq$), i.e. if $i_1<i_2<\ldots<i_k$ and
$\mu(u_{i_1}) = \mu(u_{i_2})=\ldots =\mu(u_{i_k})$ then 
$
u_{i_1}\preceq u_{i_2} \ldots \preceq u_{i_k}.
$
Using this notation we can write
 $$\quad \mu(t)=\Phi_{t}\left(\sum_{k=1}^{N} p_{k} \delta_{u_{k}}\right)=\mathcal M_{\vec p} \circ \varphi_{t} \circ \mathcal M_{\vec p}^{-1}\left(\sum_{k=1}^{N} p_{k} \delta_{u_{k}}\right)$$

$$
\nu(t)=\Psi_{t}\left(\sum_{k=1}^{N} p_{k} \delta_{u_{k}}\right)=\mathcal M_{\vec p} \circ \cdot \psi_{t}  \circ  \mathcal M_{\vec p}^{-1}\left(\sum_{k=1}^{N} p_{k} \delta_{u_{k}}\right)
$$
Or, shortly,
$$
\Phi_{t}= \mathcal M_{\vec p}\circ \varphi_{t} \circ \ \mathcal M_{\vec p}^{-1}, \ \ \Psi_{t}= \mathcal M_{\vec p} \circ \psi_{t} \circ  \mathcal M_{\vec p}^{-1}.
$$
Note that $\mathcal M_{\vec p}\circ \varphi_{t} \circ \ \mathcal M_{\vec p}^{-1}$ does not depend on choice of the order between points with equal masses. 
From this we can get property
$$
\Phi_{t}\left(\Psi_{t}(\mu)\right)=\Psi_{t}\left(\Phi_{t}(\mu)\right)=\mu.
$$
Indeed, for $\mu=\sum_{k=1}^{N} p_{k} \delta_{u_{k}}$

$$
\begin{aligned}
& \Phi_{t}\left(\Psi_{t}(\mu)\right)=\mathcal M_{\vec p}\circ  \varphi_{t} \circ \mathcal M_{\vec p}^{-1}\left(\mathcal M_{\vec p} \circ  \psi_{t} \cdot \mathcal M_{\vec p}^{-1}(\mu)\right)= \\
= & \mathcal M_{\vec p} \circ  \varphi_{t} \circ  \psi_{t}\circ  \mathcal M_{\vec p}^{-1}(\mu)=\mu .
\end{aligned}
$$
Now, for an arbitrary initial measure $\mu$ let us consider a sequence of measures
$$
\mu_{n}=\sum_{k=1}^{N_{n}} p_{k}^{(n)} \delta_{u_{k}^{(n)}}
$$
such that $\gamma_{0}\left(\mu_{n}, \mu\right) \rightarrow 0$ as $n \rightarrow \infty.$
It is known \cite{Andrey-book} that for the solutions $\mu_{t}^{\prime}, \mu_{t}^{\prime}$ to the differential equation with interaction with initial measures $\mu^{\prime}, \mu^{\prime \prime}$ the following inequality holds

$$
\sup _{[0,T]} \gamma_{0}\left(\mu_{t}^{\prime}, \mu_{t}^{\prime \prime}\right) \leq C \gamma_{0}\left(\mu^{\prime}, \mu^{\prime \prime}\right).
$$
Using notation of solution as $\Phi_{t}(\mu)=\mu_{t}$, we get the following continuity of the mapping $\Phi_{t}:$
$$
\sup _{[0, T]} \gamma_{0}\left(\Phi_{t}\left(\mu^{\prime}\right), \Phi_{t}\left(\mu^{\prime \prime}\right)\right) \leq C \gamma_{0}\left(\mu^{\prime}, \mu^{\prime \prime}\right).
$$
Analogously, the same property holds for the mapping $\Psi_{t}$. We will use this property in order to prove $\Phi_{t}\left(\Psi_{t}(\mu)\right)=\Psi_{t}\left(\Phi_{t}(\mu)\right)=\mu$ for arbitrary $\mu$. We already proved this property for $\mu_{n}$. Now, for arbitrary $\varepsilon>0$ we can find $\mu_{n}$ such that

$$
\begin{aligned}
& \sup _{[0, T]} \gamma_{0}\left(\Phi_{t}\left(\Psi_{t}(\mu)\right), \mu\right) \leq \sup _{[0, T]} \gamma_{0}\left(\Phi_{t}\left(\Psi_{t}(\mu)\right), \Phi_{t}\left(\Psi_{t}\left(\mu_{n}\right)\right)\right) \\
+ & \sup _{[0, T]} \gamma_{0}\left(\Phi_{t}\left(\Psi_{t}\left(\mu_{n}\right)\right), \mu_{n}\right)+\gamma_{0}\left(\mu_{n}, \mu\right)= \\
= & \sup _{[0, T]} \gamma_{0}\left(\Phi_{t}\left(\Psi_{t}(\mu)\right), \Phi_{t}\left(\Psi_{t}\left(\mu_{n}\right)\right)\right)+\gamma_{0}\left(\mu_{n,}, \mu\right)<\varepsilon.
\end{aligned}
$$
From this we conclude $\Phi_{t}\left(\Psi_{t}(\mu)\right)=\mu$.

\end{proof}

Note that the continuity of $\Phi$ follows from the known inequality that we used in the proof:
$$
\sup _{[0,T]} \gamma_{0}\left(\mu_{t}^{\prime}, \mu_{t}^{\prime \prime}\right) \leq C \gamma_{0}\left(\mu^{\prime}, \mu^{\prime \prime}\right).
$$

\section{Average of braid invariants as a functional of initial measure}
In the previous section we proved that the measure-valued solutions  $\{\Phi_t\}_{t\in \mathbb{R}}$ to the equation with interaction form a flow of continuous mappings on the space of probability measures $\mathfrak{ M_0}:$
$$
\Phi_0 = id, \ \Phi_t\circ\Phi_s = \Phi_{t+s}.
$$
Now we consider a braid that is formed by trajectories of particles driven by a differential equation with interaction: for $u_i \in \mathbb R^2,$ $i=1,\ldots, n,$
$$
\beta(\vec u, t): = (x(u_1,t), \ldots, x(u_n,t)), \ t\in [0, T].
$$
We will use the ergodic theorem to analyze limiting behaviour of topological invariants of a braid $\beta,$ (with $T\to \infty)$. As we mentioned in the previous section, the $n-$point motion $(x(u_1,t), \ldots, x(u_n,t)), \ t\in [0, T] $ does not satisfy evolutionary property, but the measure-valued solution $\mu_t, \ t\geq0$ forms a dynamical system in $\mathfrak{M}_0.$  This gives us an idea to represent the average of the invariants of a brad as a functional of the initial measure $\mu_t=\Phi_t(\mu_0)$ and then find the conditions under which it is possible to apply the ergodic theorem. 

Denote the linking number between trajectories which start from the points $u_{k_1}$ and $u_{k_2}$ as $\varphi(u_{k_1}, u_{k_2}, t)$.
Trivially, let $\varphi(u, u, t)=0$.
We can use the complex representation of linking number,
\begin{align*}
\varphi(u, v, t) = \frac{1}{2\pi i}\int^{t}_{0}\frac{{\rm d} z_2-z_1}{z_2-z_1},
\end{align*}
where $z_1 = x_1(u, t)+ix_2(u, t),\, z_2 = x_1(v, t)+ix_2(v, t)$.
From this we write the linking number as an integral with respect to time of some function $F$ from two trajectories:
\begin{align*}
\varphi(u, v, t) &= \frac{1}{2\pi i}\int^{t}_{0}\frac{a_1(x(v, s), \mu_s)- a_1(x(u, s), \mu_s) + i\left(a_2(x(v, s), \mu_s) - a_2(x(u,s), \mu_s)\right)}{z_2-z_1}{\rm d}s\\
&= : \int^{t}_{0} F(x(u,s), x(v,s), \mu_s){\rm d}s.
\end{align*}
Then the average linking number we can write as a functional of the initial measure $\mu_0.$ Indeed, 
\begin{align*}
  \int_{\mathbb{R}^2}\int_{\mathbb{R}^2}\varphi(u, v, t)\mu_0({\rm d}u)\mu_0({\rm d}v)
  = \int^{t}_{0}\int_{\mathbb{R}^2}\int_{\mathbb{R}^2}F(u, v, \mu_s)\mu_s({\rm d}u)\mu_s({\rm d}v){\rm d}s.
\end{align*}
Denoting by 
$$
G(\mu_s) := \int_{\mathbb{R}^2}\int_{\mathbb{R}^2}F(u, v, \mu_s)\mu_s({\rm d}u)\mu_s({\rm d}v),
$$
the average of the linking number we represent as
$$ 
  \int_{\mathbb{R}^2}\int_{\mathbb{R}^2}\varphi(u, v, t)\mu_0({\rm d}u)\mu_0({\rm d}v)=\int_0^t G(\Phi_s(\mu_0))ds.$$

Next, the general Vassiliev invariant of the braid is represented by a linear combination of integrals of $\varphi(u,v,t)$. So, the average Vassiliev invariant can be represented as a function from the initial measure $\mu_0.$

In the following sections, we will find conditions under which we can apply the ergodic theorem

\section{Liouville's Theorem for deterministic differential equations with interaction}
In this section, we will find the condition on the coefficient of the differential equation when the solution to the equation preserves a volume. We will start from the case of usual solution $x(\cdot, t)$ and will show when it preserves Lebesgue measure on $\mathbb R^2.$ 
The Liouville's theorem in Hamilton's equations asserts that the phase flow preserves phase volume\cite{Arnold-book}.
Extensions of Liouville's theorem to stochastic systems can be found in \cite{Ryogo-stoch., Kunita}.
In this section, we extend Liouville's theorem to deterministic differential equations with interaction.

\subsection{The classical Liouville's theorem}
\label{L-T}

In this subsection, we investigate the deterministic differential equations with interaction (\cite{Andrey-1-2003}-\cite{Andrey-book}) in $\mathbb{R}^2$:
\begin{equation}
\begin{cases}
{\rm d}x(u,t)=a(x(u,t),\mu_t){\rm d}t,\\	
x(u,0)=u,\,\,\,\, u\in \mathbb{R}^2,\\
\mu_t=\mu_0\circ x(\cdot,t)^{-1}.
\end{cases}
\end{equation}
Measure $\mu_0$ on $\mathbb{R}^2$ is an initial distribution of mass of particles in media.
Coefficient $a(x,\mu)$ depends on spatial variable and measure which describes the mass of particles in space.
A trajectory of particle with starting point $u\in \mathbb{R}^2$ is described by $\{x(u,t),\ t\geq0\}.$
The measure $\mu_t$ is the distribution of the mass of particles at time $t$.
It is proved that if coefficient of the equation \eqref{eq-interaction} satisfies Lipschitz condition with respect to spatial and measure-valued variables,
then there exists a unique solution to \eqref{eq-interaction} (see P78 in \cite{Andrey-book}),
and the map $x(\cdot, t): \mathbb{R}^2\rightarrow \mathbb{R}^2$ is a homeomorphism (see P87 in \cite{Andrey-book}).

Let $x'_{u},\,a'_{x}(x, \mu)$ be the Jacobian Matrices, i.e.
\begin{equation*}
  x'_{u} =
  \left(
                  \begin{array}{cc}
                    \tfrac{\partial x_1}{\partial u_1} & \tfrac{\partial x_1}{\partial u_2} \\
                    \tfrac{\partial x_2}{\partial u_1} & \tfrac{\partial x_2}{\partial u_2}  \\
                  \end{array}
                \right) (u, t), \quad
a'_{x}(x, \mu) =
\left(
                  \begin{array}{cccc}
                    \tfrac{\partial a_1}{\partial x_1} & \tfrac{\partial a_1}{\partial x_2} \\
                    \tfrac{\partial a_2}{\partial x_1} & \tfrac{\partial a_2}{\partial x_2}  \\
                  \end{array}
                \right)(x(u, t), \mu_t),
\end{equation*}
$|x'_u|$ denotes the determinant of $x'_u$
and $\trace a'_x(x, \mu)$ denotes the trace of $a'_x(x, \mu)$.
It is easy to check that the following statement holds.

\begin{lemma}
\label{lemma1}
Put
$$x'_{u,1} = (\tfrac{\partial^2 x_1}{\partial u\partial t}, \tfrac{\partial x_2}{\partial u})^{\mathrm{T}}, \quad
x'_{u,2} = (\tfrac{\partial x_1}{\partial u}, \tfrac{\partial^2 x_2}{\partial u\partial t})^{\mathrm{T}}.$$
Then
$$\frac{\rm d}{{\rm d}t}|x'_u|= |x'_{u,1}|+ |x'_{u,2}|.$$
\end{lemma}

\begin{corollary}
\label{Liouville-formula}
(Ostrogradsky-Liouville formula)
\begin{equation}\label{L-formular}
|x'_u| = \exp \int_0^t \trace a'_x(x(u, s), \mu_s){\rm d}s.
\end{equation}
\end{corollary}
\begin{proof}
By \eqref{eq-interaction},
\begin{equation*}
  x_i(u, t) = u_i + \int^t_0 a_i(x(u, s), \mu_s){\rm d}s,
  \quad i=1, 2.
\end{equation*}
Then
\begin{align*}
  \tfrac{\partial x_i}{\partial u_j}(u, t)
  = \delta_{ij}
  + \int^t_0 \sum_{l=1}^{2}\tfrac{\partial a_i}{\partial x_l}(x(u, s), \mu_s)\tfrac{\partial x_l}{\partial u_j}(u, s){\rm d}s,
  \quad i,j = 1,2,
\end{align*}
where $\delta_{ij}$ is a Kronecker function.
Thus,
\begin{align*}
  \tfrac{\partial^2 x_i}{\partial u_j\partial t}(u, t)
  =  \sum_{l=1}^{2}\tfrac{\partial a_i}{\partial x_l}(x(u, t), \mu_t)\tfrac{\partial x_l}{\partial u_j}(u, t),
  \quad i,j=1, 2.
\end{align*}
Then by Lemma \ref{lemma1},
\begin{align*}
  \frac{\rm d}{{\rm d}t}|x'_u|
  &= |x'_u| \tfrac{\partial a_1(x(u, t), \mu_t)}{\partial x_1} +|x'_u|\tfrac{\partial a_2(x(u, t), \mu_t)}{\partial x_2}\\
  &= \trace a'_x\left(x(u, t), \mu_t\right)|x'_u|.
\end{align*}
Since $ |x'_u| = 1$ at $t=0$, then
$$|x'_u| = \exp \int_0^t \trace a'_x(x(u, s), \mu_s){\rm d}s.$$
\end{proof}

The Ostrogradsky-Liouville formula is applied to find sufficient conditions under which the solution to equations with interaction \eqref{eq-interaction} preserves the measure on $\mathbb{R}^2$.

\begin{theorem}
\label{Liouville-theorem} (Liouville Theorem)
Let $m({\rm d} u) = p(u){\rm d}u $ be a measure on $\mathbb{R}^2$.
Suppose that $p\in C^1(\mathbb{R}^2)$ is a bounded and nonnegative function, then solutions to \eqref{eq-interaction} preserve the measure $m$ iff
$$\dive ( p(u)a(u, \mu_t))= 0, \,\,\forall\, t\geq 0, \, u \in \mathbb{R}^2.$$
\end{theorem}
\begin{proof}
By the homeomorphic property of the solution to \eqref{eq-interaction}, we know that $x(A, t)$ is measurable for a bounded Borel set
$A\subseteq \mathbb{R}^2$.
Define
$$V(t) = m(x(A, t)),\,\,t\geq 0,$$
where $x(A, t) = \{x(u,t), \, u\in A\}$.
Due to the continuity of $x(u, t)$ respect to $u$ and differentiability of $x(u, t)$ respect to $t$, we know that $V\in C^1(\mathbb{R}_+)$.
Compute $V(t)$ using the change of variables and Corollary \ref{Liouville-formula},
\begin{align*}
  V(t) &= \int_{x(A, t)}p(u){\rm d}u =\int_{A}p(x(u, t))|x'_u(u, t)|{\rm d}u\\
  &= \int_{A}p(x(u, t))\exp \int_0^t \trace a'_x(x(u, s), \mu_s){\rm d}s{\rm d}u.
\end{align*}

Now we compute the derivative of $V$. We first remark that the function $F(u,t)=p(x(u, t))\exp \int_0^t \trace a'_x(x(u, s), \mu_s){\rm d}s$, $u \in A$, $t\geq 0$, is differentiable in $t$ for every $u$ and its derivative $\frac{\partial F}{\partial t}(u,t)$ is bounded for $u \in A$ and $t$ from any interval $[0,T]$. By the dominated convergence theorem and the Lagrange theorem, it is easy to see that the function $V$ is differentiable and
\begin{align*}
  V'(t)= &\int_A  p'(x(u, t)) \tfrac{\partial x(u,t)}{\partial t} \exp \int_0^t \trace a'_x(x(u, s), \mu_s){\rm d}s{\rm d}u\\
         & + \int_{A}p(x(u, t))\trace a'_x(x(u, t), \mu_t)\exp \int_0^t \trace a'_x(x(u, s), \mu_s){\rm d}s{\rm d}u\\
       =&\int_A  p'(x(u, t)) a(x(u,t), \mu_t) \exp \int_0^t \trace a'_x(x(u, s), \mu_s){\rm d}s{\rm d}u\\
         & + \int_{A}p(x(u, t))\trace a'_x(x(u, t), \mu_t)\exp \int_0^t \trace a'_x(x(u, s), \mu_s){\rm d}s{\rm d}u.
\end{align*}
Thus,
\begin{align*}
  V'(0) &=\int_A\big[ p'(u)a(u, \mu_0) + p(u)\trace a'_u(u, \mu_0)\big]{\rm d} u \\
  & =\int_A \big[ \sum^2 _{i=1}\tfrac{\partial p(u)}{\partial u_i}a_i(u, \mu_0) + p(u)\sum^2_{i =1}\tfrac{\partial a_i}{\partial u_i}(u, \mu_0) \big]{\rm d} u\\
  & = \int_A \sum^{2}_{i=1}\tfrac{\partial}{\partial u_i}(p(u)a_i(u, \mu_0)){\rm d} u\\
  &= \int_A \dive(p(u)a(u,\mu_0)){\rm d} u.
\end{align*}
Hence, $V'(0)=0$ for every bounded Borel set $A\subset \mathbb{R}^2$ which implies that
$$\dive(p(u)a(u,\mu_0))=0\quad u-\mbox{a.e.}$$
Note that $\dive(p(u)a(u,\mu_0))$ is continuous, therefore
$$\dive(p(u)a(u,\mu_0))=0,
\quad \forall u \in \mathbb{R}^2.$$

For fixed $t_0>0$ we define a new flow as follows
\begin{equation}\label{new-flow}
\begin{cases}
{\rm d}y(u,t)=a(y(u,t),\nu_t){\rm d}t,\\	
y(u,0)=u,\ u\in \mathbb{R}^2,\\
\nu_t=\mu_{t_0}\circ y(\cdot,t)^{-1}.
\end{cases}
\end{equation}
Thus,
\begin{equation}\label{flow-relation}
  x(u, t_0+t) = y(x(u, t_0), t), \,\,\,\, \mu_{t_0+t} = \nu_t.
\end{equation}
By \eqref{flow-relation}, we have
$$V_{t_0}(t) = m(x(A, t_0+t)) = m(y(x(A, t_0), t)).$$
Similarly, we have
$V'(t_0) = V'_{t_0}(0) =0$ for every bounded Borel set $A\subset \mathbb{R}^2$ iff
$$\dive(p(u)a(u, \nu_0))=\dive(p(u)a(u,\mu_{t_0})) =0,\quad u \in \mathbb{R}^2.$$
Therefore, $V'(t)=0$ for every bounded Borel set $A\subset \mathbb{R}^2$ iff
$$\dive ( p(u)a(u, \mu_t))= 0, \,\, \forall\, t\geq 0,\,u\in\mathbb{R}^2.$$
This completes the proof.
\end{proof}

By Theorem \ref{Liouville-theorem}, one can easily get a sufficient condition guaranteed that solution to \eqref{eq-interaction} preserves the measure.

\begin{corollary}
Let $m({\rm d} u) = p(u){\rm d}u $ be a measure on $\mathbb{R}^2$.
Suppose that $p\in C^1(\mathbb{R}^2)$ is a bounded and nonnegative function, and
$$\dive ( p(u)a(u, \mu))= 0$$
for every probability measure $\mu$ on $\mathbb{R}^2$, all $t\geq 0$ and $u \in \mathbb{R}^2$.
Then the solution to \eqref{eq-interaction} preserves the measure $m$.
\end{corollary}
\begin{example}
 Let us consider the following equation:
$$
\left\{\begin{array}{l}
d x(u, t)=g\left(\|x(u, t)\|, \mu_{t}\right) A x(u, t) d t \\
x(u, 0)=u, \quad u \in \mathbb{R}^{2} \\
\mu_{t}=\mu_{0} \circ x(, t)^{-1}
\end{array}\right.
$$
where the function $g: \mathbb{R}_+ \times \mathfrak M_0 \rightarrow \mathbb{R}$ satisfies Lipschitz condition, $g(0)=0$, and $$A=\left(\begin{array}{ll}0 & 1 \\ -1 & 0\end{array}\right).$$
Assume that $\mu_{0}$ has density: $\mu_{0}(d u)=p(u)d u$.
Let the density $p$ be rotation-invariant. Then the measure $\mu_{0}$ is invariant, that is $\mu_{t}=\mu_{0}$ for all $t \geqslant 0.$
To check this we use Theorem \ref{Liouville-theorem} and prove that $\operatorname{div}(p(u) a(u, \mu))=0$ for all $\mu$, where 
$$
a\left(u,\mu\right)=g\left(u_{1}^{2}+u_{2}^{2}, \mu\right)\left(\begin{array}{l}
u_{2} \\
-u_{1}
\end{array}\right) \text {. }
$$
Then
$$
p(u) a\left(u, \mu\right)=\left(\begin{array}{c}
u_{2} \cdot g\left(u_{1}^{2}+u_{2}^{2}, \mu\right) \cdot p\left(u_{1}, u_{2}\right) \\
-u_{1}\cdot g\left(u_{1}^{2}+u_{2}^{2}, \mu\right) \cdot p\left(u_{1}, u_{2}\right)
\end{array}\right)
$$
and the derivative is equal to zero
$$
\begin{aligned}
& \operatorname{div}(p(u) a(u, \mu))= \\
& =u_{2} \cdot g_{2}^{\prime}\left(u_{1}^{2}+u_{2}^{2}, \mu\right) \cdot 2 u_{1} p\left(u_{1}, u_{2}\right)+p_{1}^{\prime}\left(u_{1}, u_{2}\right) u_{2} g\left(u_{1}^{2}+u_{2}^{2}, \mu\right) \\
& -u_{1} g_{1}^{\prime}\left(u_{1}^{2}+u_{2}^{2}, \mu\right) \cdot 2 u_{2} p\left(u_{1}, u_{2}\right)-p_{2}^{\prime}\left(u_{1}, u_{2}\right) \cdot u_{1} g\left(u_{1}^{2}+u_{2}^{2}, \mu\right)= \\
& =g\left(u_{1}^{2}+u_{2}^{2}, \mu\right)\left(p_{1}^{\prime}\left(u_{1}, u_{2}\right) \cdot u_{2}-p_{2}^{\prime}\left(u_{1}, u_{2}\right) u_{1}\right)=0,
\end{aligned}
$$
in the case when
$$
p_{1}^{\prime}\left(v_{1}, u_{2}\right) \cdot u_{2}-p_{2}^{\prime}\left(v_{1}, u_{2}\right) u_{1}=0 .
$$
This property is true if $p$ is rotation invariant, that is $p\left(u_{1}, u_{2}\right)=\varphi\left(u_{1}^{2}+u_{2}^{2}\right)$.

\end{example}

\subsection{Extended Liouville's theorem}
In this subsection we will prove an analog of Liouville's theorem for measure-valued flow of solutions $\Phi_t.$  Here we consider a space of finite discrete measures defined as
$$\mathfrak{K}_n = \left\{\sum_{k=1}^{n}p_k\delta_{u_k}:\ p_k>0, \sum_{k=1}^np_k=1, u_k \in \mathbb{R}^2\, \mbox{and}\
u_k\neq u_j,\  k\neq j \right\},\,\, n=1, 2,\cdots.$$
Trivially, $\mathfrak{K}_n\cap  \mathfrak{K}_m=\emptyset$ for $n \not= m$.
For any set $A \subset \mathbb{R}^2$ and $\mu=\sum_{k=1}^{n}p_k\delta_{u_k} \in \mathfrak{K}_n$, we have
\begin{align*}
  \Phi_t(\mu)(A)= \sum_{k=1}^n p_k 1_{x_{0, \mu}(A, t)^{-1}}(u_k)
  =\sum_{k=1}^n p_k 1_{A}(x_{0, \mu}(u_k, t)).
\end{align*}
Thus,
\begin{equation}\label{eq-map_phi}
\Phi_t(\mu) = \sum_{k=1}^n p_k \delta_{x_{0, \mu}(u_k, t)},
\end{equation}
which implies that $\Phi_t(\mu) \in \mathfrak{K}_n$ for all $\mu \in \mathfrak{K}_n$.

For
$(u_1, \cdots, u_n)\in (\mathbb{R}^2)^n$, $u_k \neq u_j,\,k\neq j$,
let
\begin{equation}
\label{new-measure}
{\rm d}m_n = p(u_1,\cdots,u_n){\rm d}u_1\cdots {\rm d}u_n
\end{equation}
be a measure on $(\mathbb{R}^2)^n$,
where $p \in C^1(\mathbb{R}^{2n})$ is a bounded nonnegative function.
For fixed weights $p_k>0,\,1\leq k\leq n$ and $\sum_{k=1}^np_k=1$, we define a measure $\mathfrak{m}_n$ on $\mathfrak{K}_n$ as the image of $m_n$ under the map
\[
  (u_1,\dots,u_n)\mapsto \sum_{ k=1 }^{ n } p_k \delta_{u_k}.
\]
One can interpret the measure $\mathfrak{m}_n$ as a distribution in $\mathfrak{K}_n$ of a random point measure $\sum_{ k=1 }^{ n } p_k \delta_{\xi_k}$, where $\{\xi_k\}_{k=1}^n$ are random variables with joint distribution $m_n.$

Similarly to the proof of Theorem \ref{Liouville-theorem}, we get the following theorem.

\begin{theorem}
\label{Liou-th-for-k_n}
Let $\mathfrak{m}_n$ be defined as above.
Assume that for all $u_1, \cdots, u_n \in \mathbb{R}^2,\,\mu \in \mathfrak{M}_0$,
$$\dive_k ( p(u_1,\cdots,u_n)a(u_k, \mu))= 0,\, 1\leq k \leq n,$$
then $\Phi_t$ preserves the measure $\mathfrak{m}_n$.
\end{theorem}
\begin{proof}
According to~\eqref{eq-map_phi} it is enough to show that the measure $m_n$ is preserved under the map $(x(u_1,t),\dots,x(u_n,t))$.
Let $A=A_1\times\cdots\times A_n \subseteq (\mathbb{R}^2)^n$ be a bounded Borel set.
Define
$$V(t) = m_n(x(A_1, t)\times \cdots \times x(A_n, t)),\,\,t\geq 0,$$
where $x(A_k, t) = \{x(u_k,t), \, u_k\in A_k \subseteq \mathbb{R}^2\}, \,1\leq k \leq n$.
Compute $V(t)$ using the change of variables,
\begin{align*}
  V(t) &= \int_{x(A_1, t)}\cdots \int_{x(A_n, t)}p(u_1,\cdots,u_n){\rm d}u_1\cdots {\rm d}u_n\\
  &=\int_{A_1}\cdots \int_{A_n}
      p(x(u_1, t),\cdots,x(u_n, t)) \prod^{n}_{j=1}|x'_{u_j}(u_j, t)|
      {\rm d}u_1\cdots {\rm d}u_n.
\end{align*}

Compute the derivative of $V$ using the Corollary \ref{Liouville-formula}
\begin{align*}
  V'(t)= & \int_{A_1}\cdots \int_{A_n}
  \sum_{k=1}^{n}p'_k(x(u_1, t),\cdots,x(u_n, t))a(x(u_k,t), \mu_t) \prod^{n}_{j=1}|x'_{u_j}(u_j, t)|
  {\rm d}u_1\cdots {\rm d}u_n \\
  &+ \int_{A_1}\cdots \int_{A_n}
     \sum_{k=1}^{n} p(x(u_1, t),\cdots,x(u_n, t))
     \prod^{n}_{j=1}|x'_{u_j}(u_j, t)|
     \trace a'_{u_k, \mu}(x(u_k, t), \mu_t)
      {\rm d}u_1\cdots {\rm d}u_n.
\end{align*}
Thus,
\begin{align*}
  V'(0) =&\int_{A_1}\cdots \int_{A_n}
  \sum_{k=1}^{n}p'_k(u_1,\cdots,u_n)a(u_k, \mu_0)
  {\rm d}u_1\cdots {\rm d}u_n \\
  &+ \int_{A_1}\cdots \int_{A_n}
     \sum_{k=1}^{n} p(u_1,\cdots,u_n)
     \trace a'_{u_k, \mu}(u_k, \mu_0)
      {\rm d}u_1\cdots {\rm d}u_n\\
  =&\int_{A_1}\cdots \int_{A_n}\sum_{k=1}^{n}
  [p'_k(u_1,\cdots,u_n)a(u_k, \mu_0)+p(u_1,\cdots,u_n)
     \trace a'_{u_k, \mu}(u_k, \mu_0)]{\rm d}u_1\cdots {\rm d}u_n\\
  =&\int_{A_1}\cdots \int_{A_n}\sum_{k=1}^{n}
  \dive_k(p(u_1,\cdots,u_n)a(u_k, \mu_0)){\rm d}u_1\cdots {\rm d}u_n
\end{align*}
Hence, $V'(0)=0$ for all bounded Borel set $A\in (\mathbb{R}^2)^n$ which implies that
 $$\sum_{k=1}^{n}\dive_k(p(u_1,\cdots,u_n)a(u_k, \mu_0))=0\,\, a.e.$$
for $u_k\in \mathbb{R}^2,\,1\leq k \leq n$.
Due to the continuity of $\dive_k(p(u_1,\cdots,u_n)a(u_k, \mu_0))$,
$V'(0) = 0$ for all $A\in (\mathbb{R}^2)^n$ iff
$$\sum_{k=1}^{n}\dive_k(p(u_1,\cdots,u_n)a(u_k, \mu_0))=0,\,\, \forall u_k \in \mathbb{R}^2,\,1\leq k \leq n.$$
Therefore, if for all $u_1,\cdots,u_n \in \mathbb{R}^2,\,\mu \in \mathfrak{M}_0$,
$$\dive_k(p(u_1,\cdots,u_n)a(u_k, \mu))=0,\,\, 1\leq k \leq n,$$
$V'(0) = 0$ for all $A\in (\mathbb{R}^2)^n$.

Using the new flow \eqref{new-flow} similarly,
we obtain that if for all $u_1,\cdots,u_n \in \mathbb{R}^2,\,\mu \in \mathfrak{M}_0$,
$$\dive_k(p(u_1,\cdots,u_n)a(u_k, \mu))=0,\,\, 1\leq k \leq n,$$
$V'(t) = 0$ for all $t\geq 0,\,A\in (\mathbb{R}^2)^n$.
\end{proof}

\begin{example}{(Equal weights)}
A measure $M$ on $\mathfrak{K}_n$ is defined as follows.
Let $p\in C^1(\mathbb{R}^2)$ be a bounded and nonnegative function, $A\subseteq \mathfrak{K}_n$ and
$$M(A)=\int_{\mathbb{R}^2}\cdots\int_{\mathbb{R}^2}p(u_1)\cdots p(u_n)1_{A}\left(\tfrac{1}{n}\sum^{n}_{j=1}\delta_{u_k}\right){\rm d}u_1\cdots{\rm d}u_n.$$
If for all $\mu \in \mathfrak{M}_0$, $\dive p(\cdot)a(\cdot, \mu)=0$, then
$$M\circ\Phi_t^{-1}=M.$$
\end{example}

\begin{example}
  A measure $M_\lambda$ on $\bigcup^{\infty}_{n=1}\mathfrak{K}_n$ is defined as follows.
  Let $p\in C^1(\mathbb{R}^2)$ be a nonnegative function with $\int_{\mathbb{R}^2}p(u){\rm d}u=1$, $A=\bigcup^{\infty}_{n=1}A_n\subseteq\bigcup^{\infty}_{n=1}\mathfrak{K}_n$ with $A_n \subseteq \mathfrak{K}_n$ and
$$M_\lambda(A)=
\sum_{n=1}^{\infty}e^{-\lambda}\tfrac{\lambda^n}{n!}
\int_{\mathbb{R}^2}\cdots\int_{\mathbb{R}^2}p(u_1)\cdots p(u_n)1_{A_n}\left(\tfrac{1}{n}\sum^{n}_{j=1}\delta_{u_k}\right){\rm d}u_1\cdots{\rm d}u_n.$$
The interpretation of this measure is as follows. Assume that $\xi_1,\,\cdots, \xi_n, \cdots$ are i.i.d. random vectors in $\mathbb{R}^2$ with probability density $p$, $\zeta$ is a Poissonian random variable independent of $\{\xi_n\}^\infty_{n=1}$ with parameter $\lambda$. Then $M_\lambda$ can be associated with a distribution of the measure
$$\tfrac{1}{\zeta}\sum^{\zeta}_{k=1}\delta_{\xi_k}.$$
If for all $\mu \in \mathfrak{M}_0$, $\dive p(\cdot)a(\cdot, \mu)=0$, then
$$M_\lambda\circ\Phi_t^{-1}=M_\lambda.$$

\end{example}

\section{Asymptotic behaviour of invariants}
\subsection{Example: asymptotic of rotation number}
We start with an example of an asymptotic behaviour of rotation number for a simple equation:
\begin{equation}\label{ex4-rotation}
\begin{cases}
{\rm d}x(u, t)=g(\|x(u, t)\|)Ax(u, t){\rm d}t,\\
x(u, 0) =u, \,\,\,\, u\in \mathbb{R}^2,\\
g(0)=0,
\end{cases}
\end{equation}
where function $g: \mathbb{R}_+\rightarrow \mathbb{R}$ is a scale function satisfies Lipschitz condition,
matrix $A$ is a generator of rotation
$$A=\left(
      \begin{array}{cc}
        0 & 1 \\
        -1 & 0 \\
      \end{array}
    \right).
$$
The solution to this equation is given by
$$x(u, t) = ue^{ g(\|u\|)A t}, \,\,t\geq 0,$$
where
\begin{equation*}
e^{ g(\|u\|)A t}
=
 \left(
   \begin{array}{cc}
     \cos \left(g(\|u\|)t\right) & \sin \left(g(\|u\|) t \right) \\
     -\sin \left(g(\|u\|)t \right) & \cos \left(g(\|u\|)t\right) \\
   \end{array}
 \right).
\end{equation*}
Consequently,
\eqref{ex4-rotation} have invariant sets
$$\{ u\in \mathbb{R}^2; \|u\|=R, R\geq 0 \}.$$
It is easy to check that the winding number of $x(u, t)$ around the origin is
$\tfrac{g(R)t}{2\pi}$.
Thus, the angle of rotation at the trajectory which starts from the point $u_1$ (or $u_2$) around the trajectory which starts from $u_2$ (or $u_1$) is $\varphi(u_1, u_2, t)$
$$\lim\limits_{t\rightarrow\infty} \tfrac{\varphi(u_1, u_2, t)}{t}=g(\|u_1\|\vee \|u_2\|),$$
where $\|u_1\|\vee \|u_2\| = \max\{ \|u_1\|, \|u_2\|\}$, $u_1,\,u_2\in \mathbb{R}^2$.
Hence, the average number of rotations
\begin{align*}
&\lim\limits_{t\rightarrow\infty}
\tfrac{1}{t} \int_{\mathbb{R}^2}\int_{\mathbb{R}^2}\varphi(u_1, u_2, t)
\mu_0({\rm d}u_1)\mu_0({\rm d}u_2)\\
=&\int_{\mathbb{R}^2}\int_{\mathbb{R}^2}g(\|u_1\|\vee \|u_2\|)
\mu_0({\rm d}u_1)\mu_0({\rm d}u_2).
\end{align*}
For convenience, denote
$$\omega(u_1, u_2) = g(\|u_1\|\vee \|u_2\|) ,$$
then
\begin{align}\label{varphi(u_1, u_2)}
\varphi(u_1, u_2, t) = \omega(u_1, u_2)t+o(t), \quad t\rightarrow\infty.
\end{align}

Generally, let us consider a braid with $n-$strings
$$x(u_1, t), \cdots, x(u_n, t),\,\,n=3,\cdots.$$
Define
$$\triangle_k(t)=\left\{ (t_1, \cdots, t_k)|\, 0\leq t_1\leq\cdots\leq t_k \leq t\right\},\,\, k=2,\cdots, n.$$
The invariants of the braid can be written in terms of multiple integrals:
$$m_k(t) = \int_{\Delta_k(t)}{\rm d}\varphi(u_{j_1}, u_{j'_1}, t_1){\rm d} \varphi(u_{j_2}, u_{j'_2}, t_2)\cdots {\rm d} \varphi(u_{j_k}, u_{j'_k}, t_k),\,\, k=2,\cdots, n.$$

\begin{theorem}
The multiple integrals $m_k$ of the braid satisfy the following limit behavior
\begin{equation}\label{m_k(t)}
\lim\limits_{t\rightarrow\infty} \frac{m_k(t)}{t^k}
=\frac{1}{k!}\omega(u_{j_1}, u_{j'_1})\cdots\omega(u_{j_k}, u_{j'_k})
,\,\, k=2,\cdots, n.
\end{equation}
\end{theorem}
\begin{proof}
When start with $k=2$, then
\begin{align*}
\lim\limits_{t\rightarrow\infty}\tfrac{m_2(t)}{t^2}
=&\lim\limits_{t\rightarrow\infty}\tfrac{1}{t^2} \int^{t}_{0}\int_{0}^{t_2}{\rm d}\varphi(u_1, u'_1, t_1) {\rm d} \varphi(u_2, u'_2, t_2)\\
=&\lim\limits_{t\rightarrow\infty}\tfrac{1}{t^2} \int^{t}_{0}\varphi(u_1, u'_1, s) {\rm d} \varphi(u_2, u'_2, s)\\
=&\lim\limits_{t\rightarrow\infty}\tfrac{1}{t^2} \int^{t}_{1}s\tfrac{\varphi(u_1, u'_1, s)}{s}  \varphi'(u_2, u'_2, s){\rm d}s.
\end{align*}
By continuous Toeplitz theorem and \eqref{varphi(u_1, u_2)}, we have
\begin{align*}
&\lim\limits_{t\rightarrow\infty}\tfrac{1}{t^2} \int^{t}_{1}s\left|\tfrac{\varphi(u_1, u'_1, s)}{s}-\omega(u_1, u'_1) \right| |\varphi'(u_2, u'_2, s)|{\rm d}s\\
&\leq L \lim\limits_{t\rightarrow\infty}\tfrac{1}{t^2} \int^{t}_{1}s\left|\tfrac{\varphi(u_1, u'_1, s)}{s}-\omega(u_1, u'_1) \right| {\rm d}s =0,
\end{align*}
where $L= \sup\limits_{s\geq 1}|\varphi'(u_2, u'_2, s)|$.
Thus,
\begin{align*}
&\lim\limits_{t\rightarrow\infty}\tfrac{1}{t^2} \int^{t}_{1}s\tfrac{\varphi(u_1, u'_1, s)}{s}  \varphi'(u_2, u'_2, s){\rm d}s\\
=&\omega(u_1, u'_1)   \lim\limits_{t\rightarrow\infty}\tfrac{1}{t^2} \int^{t}_{1}s\varphi'(u_2, u'_2, s){\rm d}s\\
=&\omega(u_1, u'_1)   \lim\limits_{t\rightarrow\infty}\tfrac{1}{t^2} \left(t\varphi(u_2, u'_2, t)-\varphi(u_2, u'_2, 1) -\int^{t}_{1}\varphi(u_2, u'_2, s){\rm d}s\right)\\
=&\omega(u_1, u'_1)   \lim\limits_{t\rightarrow\infty}\left(\tfrac{1}{t} \varphi(u_2, u'_2, t) -\tfrac{1}{t^2}\int^{t}_{1}\varphi(u_2, u'_2, s){\rm d}s\right).
\end{align*}
Using L'Hospital theorem,
\begin{align*}
\lim\limits_{t\rightarrow\infty}\tfrac{1}{t^2}\int^{t}_{1}\varphi(u_2, u'_2, s){\rm d}s=\lim\limits_{t\rightarrow\infty}\tfrac{1}{2t}\varphi(u_2, u'_2, t).
\end{align*}
Then, by \eqref{varphi(u_1, u_2)}, we have
\begin{align*}
&\omega(u_1, u'_1)   \lim\limits_{t\rightarrow\infty}\left(\tfrac{1}{t} \varphi(u_2, u'_2, t) -\tfrac{1}{t^2}\int^{t}_{1}\varphi(u_2, u'_2, s){\rm d}s\right)\\
=&\omega(u_1, u'_1)\lim\limits_{t\rightarrow\infty}\tfrac{1}{2t}\varphi(u_2, u'_2, t)
=\tfrac{1}{2}\omega(u_1, u'_1)\omega(u_2, u'_2).
\end{align*}

For $k=3$, \eqref{m_k(t)} is represented by
\begin{align*}
\lim\limits_{t\rightarrow\infty}\tfrac{m_3(t)}{t^3}
=&\lim\limits_{t\rightarrow\infty}\tfrac{1}{t^3} \int^{t}_{1}\int^{t_3}_{1}\int_{1}^{t_2}{\rm d}\varphi(u_1, u'_1, t_1) {\rm d} \varphi(u_2, u'_2, t_2){\rm d} \varphi(u_3, u'_3, t_3)\\
=&\lim\limits_{t\rightarrow\infty}\tfrac{1}{t^3} \int^{t}_{1} (t_3)^2\left(\tfrac{1}{(t_3)^2}\int^{t_3}_{1}\int_{1}^{t_2}{\rm d}\varphi(u_1, u'_1, t_1) {\rm d} \varphi(u_2, u'_2, t_2)\right){\rm d} \varphi(u_3, u'_3, t_3)\\
=&\tfrac{1}{2}\omega(u_1, u'_1)\omega(u_2, u'_2)\lim\limits_{t\rightarrow\infty}\tfrac{1}{t^3} \int^{t}_{1} s^2\varphi'(u_3, u'_3, s){\rm d}s\\
=&\tfrac{1}{2}\omega(u_1, u'_1)\omega(u_2, u'_2)\lim\limits_{t\rightarrow\infty}\tfrac{1}{t^3}
\left(
t^2\varphi(u_3, u'_3, t)-\varphi(u_3, u'_3, 1)-
\int^{t}_{1}2s \varphi(u_3, u'_3, s){\rm d}s\right)\\
=&\tfrac{1}{2}\omega(u_1, u'_1)\omega(u_2, u'_2)\left(\omega(u_3, u'_3)-\tfrac{2}{3}\omega(u_3, u'_3) \right)\\
=&\tfrac{1}{6}\omega(u_1, u'_1)\omega(u_2, u'_2)\omega(u_3, u'_3).
\end{align*}

By mathematical induction, one can prove \eqref{m_k(t)}.

\end{proof}


\subsection{The average invariants when the weights are equal}
In this subsection, we consider the general case of the equation with interaction with discrete initial measure with equal weights, i.e.
$$\mu_0 \in  \mathfrak{\mathring{K}}_n = \left\{ \sum_{k=1}^{n}\tfrac{1}{n}\delta_{u_k}: u_k \in \mathbb{R}^2, u_k\neq u_j, k\neq j \right\}.$$
Define maps $\Phi_n: \mathfrak{\mathring{K}}_n\times [0, +\infty)\mapsto \mathfrak{\mathring{K}}_n$ with
$$\Phi_n(\mu_0, t)=\mu_t, \,\, t\geq 0, n=1,2\cdots,$$
where $\mu_t$ is the measure-valued solution to the equation with interaction \ref{eq-interaction}:
\begin{equation}
\begin{cases}
{\rm d}x(u,t)=a(x(u,t),\mu_t){\rm d}t,\\	
x(u,0)=u,\,\,\,\, u\in \mathbb{R}^2,\\
\mu_t=\mu_0\circ x(\cdot,t)^{-1}.
\end{cases}
\end{equation}
Then $\Phi_n$ satisfies the semi-group property
$$\Phi_n\left(\Phi_n(\mu_0, s), t\right)=\Phi_n(\mu_0, t+s), \,\, t,s\geq 0.$$
By Lemma \ref{phi-t-homeomorphism}, $\Phi_n(\cdot, t)$ is a homeomorphism on $\mathfrak{\mathring{K}}_n$.

For measure $\nu^{\otimes n}$ on $(\mathbb{R}^2)^{n}$, define the measure $\Theta_n$ on $\mathfrak{\mathring{K}}_n$ as the image of $\nu^{\otimes n}$ under the map
$$(u_1, \cdots, u_n)\mapsto \tfrac{1}{n}\sum_{j=1}^{n}\delta_{u_k},\,\, n=1,2,\cdots$$
If the assumptions in Theorem \eqref{Liou-th-for-k_n} hold, then
$$\Theta_n\circ \Phi_n(\cdot, t)^{-1}=\Theta_n,\,\, n=1,2,\cdots.$$


By Ergodic Theorem \cite{Halmos}, for a measurable bounded function $G:\mathfrak{\mathring{K}}_n\to \mathbb R,$
\begin{equation}\label{G-limit}
 \lim\limits_{t\rightarrow\infty}\frac{1}{t}\int^{t}_{0}G(\Phi_n(\mu_0, s)){\rm d}s=
 E(G| \mathcal{I}_n), \,\, n=1,2,\cdots
\end{equation}
where $E(G| \mathcal{I}_n )$ is the conditional expectation of $G$ with respect to the $\mathcal{I}_n$ which is $\sigma$-algebra of $\Phi_n$-invariant sets with
$$E(G| \mathcal{I}_n )=\int_{\mathcal{I}_n}G{\rm d} \mu_0.$$

It is known that Vasil'ev's invariants can be expressed via the integral Knotsevich invariants. For more information of Kontsevich integral for braids can be found in Section 2 of \cite{Kuznetsov2015}.

We will consider the invariants of trajectories that represented as multiple integrals:
$$m_k(t)=\underbrace{\int_{\mathbb{R}^2}\cdots\int_{\mathbb{R}^2}}_{2k} \int_{\triangle_k(t)}{\rm d}\varphi(u_{1}, u'_{1}, t_1)\cdots {\rm d}\varphi(u_{k}, u'_{k}, t_k)
\mu_0({{\rm d}u_{1}})\mu_0({{\rm d}u'_{1}})\cdots \mu_0({{\rm d}u_{k}})\mu_0({{\rm d}u'_{k}}).$$
For the invartiants $m_k(t)$ we have the following asymptotic behaviour.
\begin{theorem}
Choose $\mu_0 = \tfrac{1}{n}\sum_{k=1}^{n}\delta_{u_k}$ in equations with interaction \eqref{eq-interaction}.
Assume that for all $u_1, \cdots, u_n \in \mathbb{R}^2,\,\mu \in \mathfrak{\mathring{K}}_n$,
$$\dive_k \left( p(u_1,\cdots,u_n)a(u_k, \mu)\right)= 0,\, 1\leq k \leq n.$$
Then the average invariants
\begin{align}\label{G-ave}
\lim\limits_{t\rightarrow\infty}\frac{m_k(t)}{t^k}
=\frac{1}{k!}(E(G| \mathcal{I}_n))^k, \,\, k=1,2,\cdots.
\end{align}
\end{theorem}
\begin{proof}
From our considerations before the theorem the asymptotic \eqref{G-ave} follows for $k=1$.
To prove the statement for the general $k$, we use the method of mathematical induction. Assume that the asymptotic \eqref{G-ave} holds for $k=1,\ldots, n.$ Then for $k=n+1$ we can represent
$$
m_{n+1}(t) =\int^{t}_{0}\int_{\mathbb{R}^2}\int_{\mathbb{R}^2}m_n(s){\rm d}\varphi(u_{1}, u'_{1}, s)\mu_0({{\rm d}u_{1}})\mu_0({{\rm d}u'_{1}}){\rm d}s.
$$
From  this we get
\begin{align*}
&\lim\limits_{t\rightarrow\infty}\frac{m_{n+1}(t)}{t^{n+1}}\\
=&\int^{t}_{0}\int_{\mathbb{R}^2}\int_{\mathbb{R}^2}s^n \frac {m_n(s)}{s^n}{\rm d}\varphi(u_{1}, u'_{1}, s)\mu_0({{\rm d}u_{1}})\mu_0({{\rm d}u'_{1}}){\rm d}s\\
=&\frac{1}{t^{n+1}}\int^{t}_{0}\int_{\mathbb{R}^2}\int_{\mathbb{R}^2}s^n \left(\frac {m_n(s)}{s^n}-E(G| \mathcal{I}_n) \right){\rm d}\varphi(u_{1}, u'_{1}, s)\mu_0({{\rm d}u_{1}})\mu_0({{\rm d}u'_{1}}){\rm d}s  \\
+&\frac{1}{t^{n+1}}\int^{t}_{0}\int_{\mathbb{R}^2}\int_{\mathbb{R}^2}s^n E(G| \mathcal{I}_n) {\rm d}\varphi(u_{1}, u'_{1}, s)\mu_0({{\rm d}u_{1}})\mu_0({{\rm d}u'_{1}}){\rm d}s.
\end{align*}
By continuous Toeplitz theorem for the first summand we have the upper bound
\begin{align*}
&\lim\limits_{t\rightarrow\infty}\tfrac{1}{t^{n+1}} \int^{t}_{1}s^n\left|\tfrac{m_n( s)}{s^n}-E(G| \mathcal{I}_n) \right| |\varphi'(u_2, u'_2, s)|{\rm d}s\\
&\leq L \lim\limits_{t\rightarrow\infty}\tfrac{1}{t^2} \int^{t}_{1}s^n\left|\tfrac{m_n( s)}{s^n}-E(G| \mathcal{I}_n) \right| {\rm d}s =0.
\end{align*}
For the second summand, using integrating by parts and  L'Hospital theorem we get
\begin{align*}
& \lim\limits_{t\rightarrow\infty}\tfrac{1}{t^{n+1}} \int_{\mathbb{R}^2}\int_{\mathbb{R}^2} \int^{t}_{1}s^n\varphi'(u_2, u'_2, s)
{\rm d}s
\mu_0({{\rm d}u_{2}})\mu_0({{\rm d}u'_{2}})\\
=& \lim\limits_{t\rightarrow\infty}\tfrac{1}{t^{n+1}}\int_{\mathbb{R}^2}\int_{\mathbb{R}^2} \left(t^{n}\varphi(u_2, u'_2, t)-\varphi(u_2, u'_2, 1) -\int^{t}_{1}ns^{n-1}\varphi(u_2, u'_2, s){\rm d}s\right)\mu_0({{\rm d}u_{2}})\mu_0({{\rm d}u'_{2}})\\
=&  \lim\limits_{t\rightarrow\infty}\int_{\mathbb{R}^2}\int_{\mathbb{R}^2}\left(\tfrac{1}{t} \varphi(u_2, u'_2, t) -\tfrac{1}{t^2}\int^{t}_{1}\varphi(u_2, u'_2, s){\rm d}s\right)\mu_0({{\rm d}u_{2}})\mu_0({{\rm d}u'_{2}})=\frac{E(G| \mathcal{I}_n)}{n+1}.
\end{align*}
This gives the prove of the statement.
\end{proof}

\section{Ergodic theorem for functionals of measure-valued processes}

In this subsection, consider the general case without heavy weights.
$\mathfrak{M}$ is a complete separable metric space of all probability measures on $\mathbb{R}^2$ with Wasserstein distance $\gamma_0$,
As before, we consider the maps $\Phi_t: \mathfrak{M} \mapsto \mathfrak{M}$ with
$$\Phi_t(\mu_0)=\mu_t, \,\, t\geq 0,$$
where $\mu_t$ be a measure-valued function determined by equation~\eqref{eq-interaction}.

We mentioned before that the family of the map $\{\Phi_t$, $t\geq 0\}$, satisfies the evolutionary property
\[
  \Phi_{t+s}=\Phi_{t}\circ\Phi_s\quad t,s \geq 0.
\]

We can write Vasiliev's invariants for braids in the flow that start from arbitrary initial measure $\mu_0$ as a functional of this measure. 
For $\mu_0\in\mathfrak{M}$, consider the equation
\begin{equation*}
\begin{cases}
{\rm d}x(u,t)=a(x(u,t),\mu_{t}){\rm d}t,\\	
x(u,0)=u,\ u\in \mathbb{R}^2,\\
\mu_{t}=\mu\circ x(\cdot,t)^{-1}.
\end{cases}
\end{equation*}
From this equation it is easy to see that for any bounded function $f$
$$
\int_{\mathbb R^2} f(x(u,t), x'(u,t))\mu_0(du) = \int_{\mathbb R^2}f(u, a(u,\mu_t) \mu_t(du).
$$
Using such representation we can write the invariants for braids in the flow with the arbitrary initial measure as
\begin{align*}
    m_k(t)=\underbrace{\int_{\mathbb{R}^2}\cdots\int_{\mathbb{R}^2}}_{2k} \int_{\triangle_k(t)}{\rm d}\varphi(u_{1}, u'_{1}, t_1)\cdots {\rm d}\varphi(u_{k}, u'_{k}, t_k)
\mu_0({{\rm d}u_{1}})\mu_0({{\rm d}u'_{1}})\cdots \mu_0({{\rm d}u_{k}})\mu_0({{\rm d}u'_{k}}),
\end{align*}
where as before, 
\begin{align*}
\varphi(u, v, t) &= \frac{1}{2\pi i}\int^{t}_{0}\frac{a_1(x(v, s), \mu_s)- a_1(x(u, s), \mu_s) + i\left(a_2(x(v, s), \mu_s) - a_2(x(u,s), \mu_s)\right)}{x_1(v, t)-x_1(u, t) + i(x_2(v, t)-x_2(u, t))}{\rm d}s\\
&= : \int^{t}_{0} F(u, v, \mu_s){\rm d}s.
\end{align*}

Now in order to apply the Ergodic theorem, we should have a compact space of measures. For this aim we consider the space $\mathfrak M_D$ of all probability measures on closed bounded domain $D$ in $\mathbb R^2$ (this means that for $\mu \in \mathfrak M_D$  $\mu (\mathbb R\setminus D) = 0.$)

\begin{lemma}
    $\mathfrak M_D$ is a compact set.
\end{lemma}
\begin{proof}
    Note that $\mathfrak M_D\subset \left(C(D)\right)^*$ is a closed ball. Indeed, for any $\mu \in \mathfrak M_D,$ $\| \mu \|= 1, $ and if $\mu_n \rightarrow \mu$ then $\mu \in \mathfrak M_D.$ So, by Banach-Alaoglu theorem, $\mathfrak M_D$ is weak$^*$-compact. 
\end{proof}

In the next subsection we investigate conditions of coefficients of an equation with interaction that guarantee that starting from a measure $\mu \in \mathfrak M_D$ the distribution of mass $\mu_t$ driven by an equation with interaction belong to $\mathfrak M_D$ for each $t.$ 

\subsection{Invariant set for a differential equation with interaction}
Let $D$ be a closed domain in $\mathbb{R}^{2}$. The aim is to find conditions on the function $a$ such that $$u \in D \Rightarrow x(u, t) \in D \text{ for all } t \geq 0.$$

First, let us recall the known result for the ordinary differential equation. 

\textbf{Theorem  } (Thm 7.2 in \cite{Ferrera}).  \textit{Let $\varphi: \mathbb{R}^{n} \rightarrow \mathbb{R}^{n}$ be a Lipschitz function. A closed set $S \subset \mathbb{R}^{n}$ is flow invariant with respect to a differential equation $x^{\prime}=\varphi(x)$ if and only if $\langle \varphi(x), \zeta\rangle \leq 0$ for every $x \in S$ and every $\zeta \in \hat{N}_{S}(x)$.}

Here $\hat{N}_{S}(x)$ is the regular normal cone defined as

\textbf{Definition} (Def. 4.12 in \cite{Ferrera}). \textit{For a set $S \subset X$ and $x_{0} \in S$, a vector $v \in X$ is normal to $S$ at $x_{0}$ in the regular sense, written as $v \in \hat{N}_{S}\left(x_{0}\right)$, if
$$
\left\langle v, x-x_{0}\right\rangle \leq o\left(\left|x-x_{0}\right|\right) \quad \text { for every } \quad x \in S \text{ as } |x-x_0|\to 0.
$$
The set of all such vectors, $\hat{N}_{S}\left(x_{0}\right)$, is called the regular normal cone.}

We can repeat the proof of theorem 7.2 with minor changes to get a similar result for the case of differential equation with interaction. The proof of this theorem uses the Euler method of approximation of the solution of a differential equation. So, our first step is the corresponding statement about the convergence of this scheme to the solution of an equation with interaction. 

Let us consider a differential equation  with interactions (\ref{eq-interaction}):
\begin{equation*}
\label{Equation}
\left\{\begin{array}{l}
d x(u, t)=a\left(x\left(u, t\right), \mu_{t}\right) d t, \\
x(u, 0)=u, u \in \mathbb{R}, \\
\mu_{t}=\mu_{0} \circ x(\cdot, t)^{-1}.
\end{array}\right.
\end{equation*}
For arbitrary partition of the interval $[0, T]$
$$
\pi=\left\{0=t_{0}<t_{1}<\ldots<t_{N}=T\right\}
$$
consider an Euler polygonal line, $x_{\pi}$, defined as follows:

\begin{equation}
\label{euler}
    \left\{\begin{array}{l}
x_{\pi}(u, 0)=u, \quad \tilde{\mu_{0}}=\mu_0, \\
x_{\pi}(u, t)=x_{\pi}\left(u, t_{i-1}\right)+\left(t-t_{i-1}\right) a\left(x_{\pi}\left(u, t_{i-1}\right), \widetilde{\mu}_{t_{i-1}}\right), \\
\quad \quad  \quad \quad\text { for } t \in\left(t_{i-1}, t_{i}\right], \quad i=1, \ldots, N \\
\tilde{\mu}_{t_{i}}=\mu_{0} \circ x_{\pi}\left(\cdot, t_{i}\right)^{-1} .
\end{array}\right.
\end{equation}
Denote $\delta(\pi)=\max _{i=1, \ldots N}\left(t_{i}-t_{i-1}\right).$

\begin{theorem} 
\label{Thm_euler_appox}
Let $a: \mathbb{R}^{2} \times \mathfrak M \rightarrow \mathbb{R}^{2}$
be a Lipschitz function with respect to both variables, i.e. there exists a constant $L$ such that
$$
\|a(u_1, \mu_1) - a(u_2, \mu_2)\|\leq L\left(\|u_1-u_2\| + \gamma(\mu_1, \mu_2)\right).
$$
Assume that $a$ is uniformly bounded, that is
$$
\|a(u, \mu)\|\leq C_1, \text{ for all } x\in \mathbb R^2 \text{ and } \mu \in \mathfrak M.  
$$
Then
$$\sup _{t \in[0, T]} \int_{\mathbb{R}^{2}}\left\|x_{\pi}(u, t)-x(u, t)\right\| \mu_{0}(d u) \rightarrow 0 \text{ as } \delta(\pi) \rightarrow 0.$$
\end{theorem}
\begin{proof} Let us  denote

$$
\alpha(u, t)=x(u, t)-{x}_{\pi}(u, t).
$$
 
By Taylor formula, for $t \in\left(t_{k-1}, t_{k}\right]$
\begin{equation}
\label{tailor}
  x(u, t)=x\left(u, t_{k-1}\right)+x^{\prime}(u, t)\left(t-t_{k-1}\right)+R\left(u, t, t_{k-1}\right)
\end{equation}
Let us estimate remaining term $R\left(u, t, t_{k-1}\right).$ By the mean value theorem, there exists a point $\theta_{u} \in\left[t, t_{k-1}\right]$ such that
$$
\begin{aligned}
& R\left(u, t, t_{k-1}\right)=x(u, t)-x\left(u, t_{k-1}\right)-x^{\prime}(u, t) \left( t-t_{k-1}\right)= \\
& =x^{\prime}\left(u, \theta_{u}\right)\left(t-t_{k-1}\right)-x^{\prime}(u, t)\left(t-t_{k-1}\right).
\end{aligned}
$$
 Taking into account that $x$ is a solution to differential equation with interaction we can continue:
$$
\begin{gathered}
R\left(u, t, t_{k-1}\right)=\left(x^{\prime}\left(u, \theta_{u}\right)-x^{\prime}(u, t)\right)\left(t-t_{k-1}\right)= \\
=\frac{a\left(x\left(u, \theta_{u}\right), \mu_{\theta_{u}}\right)-a\left(x(u, t), \mu_{t}\right)}{t-t_{k-1}}\left(t-t_{k-1}\right)^{2} .
\end{gathered}
$$
Using Lipschitz condition of the function $a$, we get
\begin{equation}
\label{|R|}
  \left|R\left(u, t, t_{k-1}\right)\right| \leq \frac{1}{t-t_{k-1}} L\left(\left\|x\left(u, \theta_{u}\right)-x(u, t)\right\|+\gamma\left(\mu_{\theta_{u}}, \mu_{t}\right)\right) \cdot\left(t-t_{k-1}\right)^{2}.  
\end{equation}

By the mean-value theorem there exist a point $\xi_u \in\left[t, \theta_{u}\right]$ such that
$$
\frac{\left\|x\left(u, \theta_{u}\right)-x(u, t)\right\|}{t-t_{k-1}}=\frac{\left\|x^{\prime}(u, \xi)\right\|\left|\theta_{u}-t\right|}{t-t_{k-1}} \leq\left|a\left(x(u, \xi_{u}), \mu_{\xi_{u}}\right)\right| \leqslant C_1,
$$
where we used uniform boundedness of the function $a.$
Now let us estimate the Wassernstein distance $\gamma\left(\mu_{\theta_{u}}, \mu_{t}\right).$
From the definition it follows that
$$
\begin{aligned}
& \frac{1}{t-t_{k-1}} \gamma\left(\mu_{\theta_{u}}, \mu_{t}\right) \leq \frac{1}{t-t_{k-1}} \int_{\mathbb{R}} \frac{\left\|x\left(u, \theta_{u}\right)-x(u, t)\right\|}{1+\left\|x\left(u, \theta_{u}\right)-x(u, t)\right\|} \mu_{0}(d u) \\
& =\frac{1}{t-t_{k-1}} \int_{\mathbb{R}} \frac{\left.\left\|x^{\prime}\left(u, \xi_{u}\right)\right\| \mid t-\theta_{u}\right|}{1+\left\|x^{\prime}\left(u, \xi_{u}\right)\right\|\left|t-\theta_{u}\right|} \mu_{0}(d u) \\
& \leq \frac{1}{t-t_{k-1}} \int_{\mathbb{R}} C_{1}\left|t-t_{k-1}\right| \mu_{0}(d u)=C_{1} .
\end{aligned}
$$

Substitute obtained estimation to \ref{|R|} and get
$$
\left|R\left(u, t, t_{k-1}\right)\right| \leq 2 L \cdot C_1 \cdot\left(t-t_{k-1}\right)^{2}
$$
 
Now we can estimate the difference between solution to (\ref{Equation}) and Euler polygonal lines built by (\ref{euler}). To do this we subtract \ref{tailor} from \ref{euler} and using the definition of $\alpha,$
$$
\begin{aligned}
& \alpha(u, t)=\alpha\left(u, t_{k-1}\right)+ \\
& \left(t-t_{k-1}\right)\left(a\left(x\left(u, t_{k-1}\right), \mu_{t_{k-1}}\right)-a\left(x_{\pi}\left(u, t_{k-1}\right), \widetilde{\mu}_{t_{k-1}}\right)\right)+ R(u,t, t_{k-1})
\end{aligned}
$$
Hence
$$
\begin{aligned}
& \|\alpha(u, t)\| \leq\left\|\alpha\left(u, t_{k-1}\right)\right\|+\left(t-t_{k-1}\right) L\left(\left\|\alpha\left(u, t_{k-1}\right)\right\|+\gamma\left(\mu_{t_{k-1},}, \tilde{\mu}_{t_{k-1}}\right)\right)+ \\
& \quad+\left(t-t_{k-1}\right)^{2} \cdot 2L\cdot C_1
\end{aligned}
$$
Since
$$
\begin{aligned}
& \gamma\left(\mu_{t_{k}}, \tilde{\mu}_{t_{k}}\right)=\inf _{\kappa \in Q\left(\mu_{t_{k}}, \tilde{\mu}_{t_{k}}\right)} \iint_{\mathbb{R}^{2} \mathbb{R}^{2}} \frac{\|u-v\|}{1+\|u-v\|} \kappa (d u, d v) \leq \\
& \leq \int_{\mathbb{R}^{2}}\left\|x\left(u, t_{k}\right)-x_{\pi}\left(u, t_{k}\right)\right\| \mu_{0}(d u),
\end{aligned}
$$
we can continue:
$$
\begin{aligned}
& \left\|\alpha\left(u, t\right)\right\| \leq\left\|\alpha\left(u, t_{k-1}\right)\right\|+\\ 
&\left(t-t_{k-1}\right) L\left(\left\|\alpha\left(u, t_{k-1}\right)\right\|+\int_{R^{2}}\left\|\alpha\left(u, t_{k-1}\right)\right\| \mu_{0}(d u)\right) +\left(t-t_{k-1}\right)^{2} \cdot 2L\cdot C_1.
\end{aligned}
$$
Now, denoting by $\bar{\alpha}(t)=\int_{\mathbb{R}^{2}}\left\|\alpha\left(u, t\right)\right\| \mu_{0}(d u)$ and by $m = 2L\cdot C_1$, we have
$$
\bar{\alpha}(t) \leq \bar{\alpha}\left(t_{k-1}\right)\left(1+\left(t-t_{k-1}\right) 2 L\right)+\left(t-t_{k-1}\right) \cdot \delta(\pi)\cdot m.
$$
If $L>0$, it follows that
$$
\begin{aligned}
\left(\bar{\alpha}(t)+\frac{\delta(\pi) m}{2 L}\right) & \leq\left(1+\left(t-t_{k-1}\right) 2 L\right)\left(\bar{\alpha}\left(t_{k-1}\right)+\frac{\delta(\pi)}{2 L} m\right) \leqslant \\
& \leq e^{2 L\left(t-t_{k-1}\right)}\left(\bar{\alpha}\left(t_{k-1}\right)+\frac{\delta(\pi) \cdot m}{2 L}\right) .
\end{aligned}
$$
From this we get
$$
\alpha(t) \leq e^{2 L t} \cdot \bar{\alpha}(0)+\frac{e^{2 L t}-1}{2 L} \delta(\pi) \cdot m.
$$
Note that $\bar{\alpha}(0)=0$, so
$$
\sup _{t \in[0, T]} \bar{\alpha}(t) \rightarrow 0 \quad \text { as } \delta(\pi) \rightarrow 0.
$$
\end{proof}

 Now let us come back to the statement about the invariant set for a differential equation with interaction.
To formulate the statement we need some notations. Denote the set of all measures concentrated on some domain $S\subset \mathbb{R}^{2}$ by $\mathcal{M}_S$. More precisely, this means that
$$
\mu \in \mathcal{M}_S\  \Rightarrow\  \mu(\mathbb R^2 \setminus S) = 0.
$$
For any domain $S\subset \mathbb{R}^{2}$ and $\varepsilon>0$ we denote by $S^\varepsilon$  the $\varepsilon-$neighborhood of the set $S,$ that is
$$
S^\varepsilon=\{x\in \mathbb R^2: \operatorname{dist}(x, S)\leq \varepsilon\}.
$$
 
\begin{theorem} Let $a: \mathbb{R}^{2} \times \mathcal{M} \rightarrow \mathbb{R}^{2}$ be a Lipschitz function on both variables and  for every $\mu\in\mathcal{M},$ $a(\cdot, \mu) \in C'_b(\mathbb R^2)$.  Let $\varepsilon>0$ and $\mu_0\in \mathcal{M}_{D}. $   A closed domain $D \subset \mathbb{R}^{2}$ is flow invariant with respect to differential equation with interactions
$$
\left\{\begin{array}{l}
d x(u, t)=a\left(x(u, t), \mu_{t}\right) d t \\
x(u, 0)=u, u \in D \\
\mu_{t}=\mu_0 \circ x(\cdot, t)^{-1}
\end{array}\right.
$$
if for   any $\mu \in \mathcal{M}_{D^\varepsilon},$   $\langle a(x, \mu), \xi\rangle \leq 0$ for every $x \in D$ and every $\xi \in \hat{N}_{D}(x)$.
\end{theorem}
\begin{proof}
For arbitrary partition of the interval $[0, T]$
$$
\pi=\left\{0=t_{0}<t_{1}<\ldots<t_{N}=T\right\}.
$$
  For given $\varepsilon$ we select such $N$ and $\delta(\pi)$ that
$$
 C^{2} T \delta(\pi) \sum_{j=1}^{\infty} j(2 L \delta(\pi))^{j}<\varepsilon,
$$
where the constant $C$ is an upper bound of the function $a$ and $L$ is the Lipschitz constant of the function $a. $  
Consider an Euler polygonal line $x_{\pi}$, defined as in (\ref{euler}). We put $x_{i}:=x_{\pi}\left(\cdot, t_{i}\right)$. 
Let us denote by $d_{D}(x)$ the distance from a point $x$ to set $D$:
$$
d_{D}(x)=\operatorname{dist}(x, D) .
$$

Since $D$ is a closed set, for each $i=0, \ldots, N$ we can choose $y_{i}(u) \in D$ such that

$$
d_{D}\left(x_{i}(u)\right)=\left\|x_{i}(u)-y_{i}(u)\right\|.
$$
Note here that $x_i - y_i \in \hat{N}_{D}\left(y_i\right).$  Indeed, for every $v\in D,$ 
$$
\begin{aligned}
&\|x_i - y_i\|^2\leq \|x_i - v\|^2 = \langle x_i -y_i+y_i -v,\ x_i -y_i +y_i-v \rangle=\\
&=\langle x_i -y_i,\ x_i -y_i  \rangle+ 2\langle x_i -y_i,\ y_i-v \rangle+ \langle y_i -v,\ y_i-v \rangle=\\
&=\|x_i - y_i\|^2- 2\langle x_i -y_i,\ v- y_i\rangle +\langle y_i -v,\ y_i-v \rangle.
\end{aligned}
$$
From this we get 
$$
2\langle x_i -y_i,\ v- y_i\rangle \leq \langle y_i -v,\ y_i-v \rangle = o(\|v-y_i\|)
$$
and by the definition this means that $x_i-y_i \in \hat{N}_{D}\left(y_i\right).$

The distance from Euler's polygon nodes to set $D$ can be estimated as

$$
d_{D}\left(x_{1}\right) \leq\left\|x_{1}-x_{0}\right\|=\left(t_{1}-t_{0}\right)\left\|a\left(x_{0}, \mu_{0}\right)\right\| \leq C t_{1} .
$$
Note that 
$$
Ct_1< C^{2} T \delta(\pi) \sum_{j=1}^{\infty} j(2 L \delta(\pi))^{j}<\varepsilon.
$$
 From this follows that $\widetilde{\mu}_1 \in \mathfrak M_{D^\varepsilon}.$ Indeed, by definition, $\widetilde{\mu}_1 =\mu_0\circ x_{\pi} (\cdot, t_1)^{-1}$ and
$$
\widetilde{\mu}_1(\mathbb R^2\setminus D^{\varepsilon}) = \mu_0 (u: x_{\pi} (u, t_1) \in \mathbb R^2\setminus D^{\varepsilon}) = \mu_0(\text{\O} )=0.
$$ 
Next,
$$
\begin{aligned}
& d_{D}^{2}\left(x_{2}\right) \leq\left\|x_{2}-y_{1}\right\|^{2}=\left\|x_{2}-x_{1}\right\|^{2}+\left|x_{1}-y_{1}\right|^{2}
+2\left\langle x_{2}-x_{1}, x_{1}-y_{1}\right\rangle=\\
& =\left\|x_{2}-x_{1}\right\|^{2}+d_{D}^{2}\left(x_{1}\right)  +2\left(t_{2}-t_{1}\right)\left\langle a\left(x_{1}, \widetilde{\mu}_{1}\right), x_{1}-y_{1}\right\rangle= \\
& =\left\|x_{2}-x_{1}\right\|^{2}+d_{D}^{2}\left(x_{1}\right)+2\left(t_{2}-t_{1}\right)\left\langle a\left(x_{1}, \widetilde{\mu}_{1}\right)-a\left(y_{1}, \widetilde{\mu}_{1}\right), x_{1}-y_{1}\right\rangle \\
& +2\left(t_{2}-t_{1}\right)\left\langle a\left(y_{1}, \widetilde{\mu}_{1}\right), x_{1}-y_{1}\right\rangle \leqslant \\
& \leq\left\|x_{2}-x_{1}\right\|^{2}+d_{D}^{2}\left(x_{1}\right)+2\left(t_{2}-t_{1}\right)\left\langle a\left(x_{1}, \widetilde{\mu}_{1}\right)-a\left(y_{1}, \widetilde{\mu}_{1}\right), x_{1}-y_{1}\right\rangle,
\end{aligned}
$$
where we used the assumption of the theorem,
$$
\left\langle a\left(y_{1}, \widetilde{\mu}_{1}\right), x_{1}-y_{1}\right\rangle \leqslant 0 \text {, since } x_{1}-y_{1} \in \hat{N}_{D}\left(y_{1}\right)  \text {, and } \widetilde{\mu}_1 \in \mathfrak M_{D^\varepsilon}.
$$
Continue the estimation of $d_{D}^{2}\left(x_{2}\right)$ :
$$
\begin{aligned}
& d_{D}^{2}\left(x_{2}\right) \leq\left(t_{2}-t_{1}\right)^{2} \left\|a\left(x_{1}, \mu_{1}\right)\right\|^{2}+d_{D}^{2}\left(x_{1}\right)+ \\
& +2\left(t_{2}-t_{1}\right)\left\|a\left(x_{1}, \mu_{1}\right)-a\left(y_{1}, \mu_{1}\right)\right\| \cdot\left\|x_{1}-y_{1}\right\| \leq \\
& \leq C^{2}\left(t_{2}-t_{1}\right)^{2}+d_{D}^{2}\left(x_{1}\right)+2\left(t_{2}-t_{1}\right) L\cdot d_{D}^{2}\left(x_{1}\right)<\varepsilon
\end{aligned}
$$
  By the same argument as before we get that  $\widetilde{\mu}_2 \in \mathfrak M_{D^\varepsilon}.$  
Analogously, 
$$
\begin{aligned}
& d_{D}^{2}\left(x_{i+1}\right) \leq C^{2}\left(t_{i+1}-t_{i}\right)^{2}+d_{D}^{2}\left(x_{i}\right)+2 L\left(t_{i+1}-t_{i}\right) d_{D}^{2}\left(x_{i}\right) \leq \\
& \leq C^{2}\left(t_{i+1}-t_{i}\right) \delta(\pi)+d_{D}^{2}\left(x_{i}\right)+2 L\cdot d_{D}^{2}\left(x_{i}\right) \cdot \delta(\pi) .
\end{aligned}
$$
Applying these inequalities for consecutive indices $i$ we get:
$$
\begin{aligned}
& d_{D}^{2}\left(x_{i+1}\right) \leq C^{2}\left(t_{i+1}-t_{0}\right) \delta(\pi)+2 L \delta(\pi)\left(d_{D}^{2}\left(x_{i}\right)+\ldots+d_{D}^{2}\left(x_{1}\right)\right) \leq \\
& \leq C^{2} T \delta(\pi)+2 L C^{2} T \delta^{2}(\pi)+2(2 L \delta(\pi))^{2}  \cdot\left(d_{D}^{2}\left(x_{i-1}\right)+\ldots+d_{D}^{2}\left(x_{1}\right)\right) \leq \ldots \\
& \leq C^{2} T \delta(\pi) \sum_{j=1}^{\infty} j(2 L \delta(\pi))^{j}.
\end{aligned}
$$
The last series is convergent provided that $\delta(\pi)<\frac{1}{2 L}$ hence, $d_{s}^{2}\left(x_{i+1}\right) \rightarrow 0$ as $\delta(\pi) \rightarrow 0$. As the polygonal line converges uniformly on $[0, T]$ to the unique solution, we have that the trajectory remains in $D$ until $t=T$. But $T$ is arbitrary, hence $D$ is flow invariant.
\end{proof}

\subsection{Application of erodic theorem}
Now we can use the known result from ergodic theory \cite{CornfeldFominSinai}.
Let $f: M\to M$ be a continuous transformation on a compact metric space. Then there exists some probability measure $\zeta$ on $M$ invariant under $f:$
$$
\zeta(E) = \zeta(f^{-1}(E)) \text{ for any measurable } E\subset M
$$

We apply this result for 
$$
x(\cdot, t)^{-1}: \mathfrak M_D \to \mathfrak M_D.
$$

So, there exists an invariant probability measure $\zeta$ on $\mathfrak M_D.$ Then, by Birkhoff-Khinchin Ergodic Theorem \cite{CornfeldFominSinai} for a continuous function 
$$
\Phi (\mu_t) = \int_{\mathbb R^2}\int_{\mathbb R^2} f(u,v,a(u,\mu_t), a(v,\mu_t))\mu_t(du)\mu_t(dv)
$$
there exists a limit
$$
\lim_{t\to\infty}\frac{1}{t}\int_0^t \Phi(\mu_t)dt= \mathbb E (\Phi(\mu_0)|I),
$$
where $I$ is the $\sigma-$algebra of invariant sets in $\mathfrak M_D.$

Next lemma gives an example of invariant measure under transformation $\Phi_t. $
\begin{lemma}
    Assume that for equation with Lipschitz coefficient $a: \mathbb{R} \times \mathfrak M_{0} \rightarrow \mathbb{R}$,
$$
\left\{\begin{array}{l}
x\left(u,t\right)=a\left(x(u, t), \mu_{t}\right) d t \\
x(u, 0)=u \in \mathbb{R}, \\
\mu_{t}=\mu_{0} \circ x(\cdot, t)^{-1}
\end{array}\right.
$$
there exists a set $A$ of invariant initial measures,
that is $A=\left\{\mu \in \mathfrak M_{0}: \Phi_{t}(\mu)=\mu, \ t \geq 0\right\}$.
Then the measure $\varkappa$ on $\mathfrak M_{0}$ of the form
$$
\varkappa=\sum_{k=1}^{N} \delta \mu_{k} \cdot p_{k}, \quad \mu_{k} \in A
$$
is invariant with respect to $\Phi_{t}$ :
$$
\varkappa \circ \Phi_{t}^{-1}=\varkappa .
$$
\end{lemma} 

\begin{proof}   
First of all note that since $\Phi_{t}$ is invertible and $\Phi_{t}^{-1}$ is the solution to differential equation with interaction with coefficient $\hat{a}=-a$, then for $\mu \in A$
$$
\Phi_{t}^{-1}(\mu)=\mu.
$$
Next, for $\mu_{k} \quad\left(\mu_{k} \in \varkappa \right)$
$$
\begin{aligned}
& \varkappa \circ \Phi_{t}^{-1}\left(\mu_{k}\right)=\varkappa\left\{\nu \in \mathfrak M_{0}: \Phi_{t}(\nu)=\mu_{k}\right)= \\
& =\varkappa\left\{\nu \in\mathfrak M_{0}: \nu=\Phi_{t}^{-1}\left(\mu_{k}\right)=\mu_{k}\right\}=\varkappa\left(\mu_{k}\right)=p_{k} .
\end{aligned}
$$

So, $\varkappa\cdot \Phi_{t}^{-1}=\varkappa$. 
\end{proof}

For the differential equation with the interaction of that describes rotation around the origin, we can specify the set of invariant measures. 
Consider the equation
\begin{equation}
\label{rotation_eq}
    \left\{\begin{array}{l}
 x(u, t)=g\left(\|x(u, t)\|, \mu_{t}\right) A x(u, t) d t, \\
x(u, 0)=u, u \in \mathbb{R}^{2} \\
\mu_{t}=\mu_{0} \circ x(\cdot, t)^{-1}
\end{array}\right.
\end{equation}
where $A=\left(\begin{array}{cc}0 & 1 \\ -1 & 0\end{array}\right).$
The next example gives us a measure $\varkappa$ on $\mathfrak M_{0}$ such that
$$
\varkappa \circ \Phi_{t}^{-1}=\varkappa.
$$

\begin{example}
    Define a measure $\varkappa$ as the probability distribution on $\mathfrak M_{0}$ of measures $\delta_{ c \xi}$, where $\xi$ is uniformly distributed random variable on the circle $S_{1}=\left\{x \in \mathbb{R}^{2}:\|x\|=1\right\},$ and $c$ is a nonegative random variable, independent of $\xi$. So $\varkappa=\textit{Law}\left(\delta_{c\xi}\right)$. 

Now we want to check that $\varkappa \circ\Phi^{-1}_{t}(\Delta)=\varkappa(\Delta)$ for any $\Delta \in B\left(\mathfrak M_{0}\right).$
For this we consider a set $\Delta=\left\{\delta_{u}, u \in A\right\}$ for some $A = [a,b]\times I,$ where $I \subset S_{1}$.
Then 
$$\varkappa(\Delta)=\mathbb{P}\left(\delta_{c\xi} \in \Delta\right)=\mathbb{P}(\xi \in I) \mathbb{P}(c\in [a,b])$$
and
$$
\begin{aligned}
\varkappa & \circ \Phi_{t}^{-1}(\Delta)=\varkappa\left\{\nu \in \mathfrak M_{0}: \Phi_{t}(\nu) \in \Delta\right\}= \\
& =\mathbb{P}\left(\  \delta_{c\xi} \in\left\{\nu: \Phi_{t}(\nu)\in\left\{\delta_{u}, u \in A\right\}\right\}\ \right)
\end{aligned}
$$
Since
$$
\Phi_{t}(\nu)=\delta_{u} \Leftrightarrow \nu=\Phi_{t}^{-1}\left(\delta_{u}\right)=\delta_{x^{-1}(u, t)},
$$
The set $\left\{\nu: \Phi_{t}(\nu) \in\left\{\delta_{u}, u\in I\right\}\right\}=\left\{\delta_{x^{-1}}(u, t), u \in I\right\}.$
So,
$$
\begin{aligned}
& \varkappa  \circ \Phi_{t}^{-1}(\Delta)=\mathbb{P}\left\{\delta_{\xi} \in\left\{\delta_{x^{-1}(u, t)}, u \in I\right\}\right\}= \\
& =\mathbb{P}\left(\xi \in x^{-1}(A, t)\right\} .
\end{aligned}
$$
From this follows that the solution to a differential equation with interactions preserves measure then
$$
\varkappa(\Delta)=\varkappa  \circ \Phi_{t}^{-1}(\Delta) \text { for } \delta=\left\{\delta_{u}, u \in A\right\} .
$$
So, we proved that
$$
\varkappa =\varkappa  \circ \Phi_{t}^{-1}.
$$

\end{example} 
In the next example for the same equation \ref{rotation_eq} we consider random initial measure and will characterize the asymptotic of the average of the winding number. 
\begin{example}

For the initial measure $\mu_{0}$ we will take a random point measure of the  form
$$
\mu_{0}=\mathbb{I}_{\{\zeta>0\}} \quad \sum_{k=1}^{\zeta} \delta_{\xi_{k}}
$$
where $\zeta$ is a Poisson random variable with parameter $\lambda>0$ and $\left\{\xi_{k}\right\}_{k\geq 1}$ is a sequence of independent random vectors in $\mathbb{R}^{2}$ with distribution $N(0, I)$.

We  apply ergodic theorem for the average of the winding number of two trajectories, $\varphi(u, v, t)$, which is given by

$$
\frac{1}{t} \int_{\mathbb{R}^{2}} \int_{\mathbb{R}^{2}} \varphi(u, v, t) \mu_{0}(d u) \mu_{0}(d v), \ t>0.
$$
Recall that we can represent it as
$$
\begin{aligned}
& \int_{\mathbb{R}^{2}} \int_{\mathbb{R}^{2}} \varphi(u, v, t) \mu_{0}(d u) \mu_{0}(d v)=\int_{0}^{t} \int_{\mathbb{R}^{2}} \int_{\mathbb{R}^{2}} F\left(u, v, \mu_{s}\right) \mu_{s}(d u) \mu_{s}(d v) d s \\
&=\int_{0}^{t} \int_{\mathbb{R}^{2}} \int_{\mathbb{R}^{2}} F\left(u_{1} v, \Phi_{s}\left(\mu_{0}\right)\right) \Phi_{s}\left(\mu_{0}\right)(d u) \Phi_{s}\left(\mu_{0}\right)(d v) d s \\
&=\int_{0}^{t} G\left(\Phi_{s}\left(\mu_{0}\right)\right) d s.
\end{aligned}
$$
Now we want to characterize the limit
$$
\lim _{t \rightarrow \infty} \frac{1}{t} \int_{0}^{t} \int_{\mathbb{R}^{2}} \int_{\mathbb{R}^{2}} F\left(u, v, \mu_{s}\right) \mu_{s}(d u) \mu_{s}(d v) d s
$$
For this, let as define a new random measure on $\mathbb{R}^{2}$ as
$$
\nu=\int_{0}^{\infty}\nu_{r} n(d r),
$$
where $n$ is a Poisson process on $\mathbb{R}_{+}$with the intensity measure with the density
$$
x e^{-x^{2} / 2} \mathbb{I}_{\{x \geqslant 0\}}
$$
and $\gamma_{r}$ is the uniform distribution on
$
S_{r}=\left\{x\in \mathbb R^2: \ \|x\|=r\right\} .
$

\begin{lemma} For any function $F: \mathbb R^2\times \mathbb R^2\times \mathbb{M}\to \mathbb R,$ continuous with respect to third variable, 
   $$
\lim _{t \rightarrow \infty} \frac{1}{t} \int_{0}^{t} \int_{\mathbb{R}^{2}} \int_{\mathbb{R}^{2}} F\left(u, v, \mu_{s}\right) \mu_{s}(d u) \mu_{s}(d v) d s=\int_{\mathbb{R}^{2}} \int_{\mathbb{R}^{2}} F(u, v, \nu) \nu(d u) \nu(d v)
$$ 
\end{lemma}

\begin{proof}

Let us prove now that
$$
\lim_{t\to\infty}\frac{1}{t} \int_{0}^{t} \mu_{s}(\Delta) d s = \nu(\Delta) \quad \text { a.s. }
$$
For this, recall that we can use the following representation for measures $\mu_{0}$ and $\nu:$
$$
\mu_{0}=\sum_{k=0}^{\varkappa} \delta_{\xi_{k}}, \quad \nu=\sum_{k=0}^{\varkappa} \nu_{\eta_{k}}
$$
where $\varkappa, \{\xi_{k}\}_{k\geq 0}, \{\eta_{k}\}_{k \geq0}$ are independent and
$$
\varkappa \sim \text{Pois}(1), \quad \xi_{k} \sim N(0, I), \quad \eta_{k} \sim\left\|\xi_{k}\right\| .
$$

We split the proof into 3 steps:\\
Step I: $\mu_{0}=\sum_{k=1}^{N} \delta_{u_{k}}$, where $u_{k}$ are nonrandom,
$
u_{k} \neq u_{j}, k \neq j.
$
In this case $\mu_{t}=\sum_{k=1}^{N} \delta_{x\left(u_{k}, t\right)}$. From the 
form of the equation
$$
d x(u, t)=g\left(\left\|x\left(u_{1}, t\right)\right\|, \mu_{t}\right) A x\left(u_{1} t\right) d t,
$$
we have $\left\|x\left(u_{k}, t\right)\right\|=\left\|u_{k}\right\|$ for all $t \geq 0$, and $\quad\{x(u, t), t>0\}=\{v:\|v\|=\|u\|\}.$
From this,
$$
\begin{aligned}
& \frac{1}{t} \int_{0}^{t} \mu_{s}(\Delta) d s=\frac{1}{t} \int_{0}^{t} \sum_{k=1}^{N} \mathbb{I}_{\left\{x\left(u_{k}, s\right) \in \Delta\right\}} d s= \\
= & \sum_{k=1}^{N} \frac{1}{t} \int_{0}^{t} \mathbb{I}_{\left\{x\left(u_{k}, s\right) \in \Delta\right\}} d s.
\end{aligned}
$$
Applying ergodic theorem to $x(u,\cdot) ,$ we get
$$
\lim _{t \rightarrow \infty} \frac{1}{t} \int_{0}^{t} \mathbb{I}_{\{x(v, s) \in \Delta\}} d s=\mathbb{P}\left(X_{\|u\|} \in \Delta_{\|u\|}\right) \text {, }
$$
where $X_{r}$ is a random variable uniformly 
distributed on $S_{r}=\left\{v \in \mathbb{R}^{2}:\|v \cdot\|=r\right\}$ 
and $\quad \Delta_{r}=\{v \in \Delta:\|v\|=r\}$\\
So,
$$
\lim _{t \rightarrow \infty} \frac{1}{t} \int_{0}^{t} \mu_{s}(\Delta) d s=\sum_{k=1}^{N} \mathbb{P}\left(X_{\left\|u_{k}\right\|} \in \Delta_{\left\|u_{k}\right\|}\right)
$$

Step II : For the measure $\mu_{0}=\sum_{k=1}^{N} \delta_{\xi_{k}}$ that concentrated in random points $(\xi_k)_{k\geq 1}^N$ similarly we have
$$
\frac{1}{t} \int_{0}^{t} \mu_{s}(\Delta) d s=\sum_{k=1}^{N} \frac{1}{t} \int_{0}^{t} \mathbb I_{\left\{x\left(\xi_{k, s} s\right) \in \Delta\right\}} d s
$$
and a.s.
$$
\lim _{t \rightarrow \infty} \frac{1}{t} \int_{0}^{t} \mathbb{I}\{x(\xi, s) \in \Delta\} d s=\left.\mathbb{P}\left(X_{\| c \|} \in \Delta_{\|u\|}\right)\right|_{u=\xi}, \text { a.s.}
$$
So,
$$
\lim _{t \rightarrow \infty} \frac{1}{t} \int_{0}^{t} \mu_{s}(\Delta) d s=\sum_{k=1}^{N} \mathbb{P}\left(X_{\left\|\xi_{k}\right\|} \in \Delta_{\left\|\xi_{k}\right\|}\mid \xi_{k}\right)
$$
when $\left\{X_{r}\right\}$ and $\left\{\xi_{k}\right\}_{k \geq 1}$ are independent.\\

Step III. Now, for the random number of points,

$$
\mu_{0}=\sum_{k=1}^{\varkappa} \delta_{\xi_{k}}
$$
we have
$$
\lim _{t \rightarrow \infty} \frac{1}{t} \int_{0}^{t} \mu_{s}(\Delta) d s=\sum_{k=1}^{\infty} \mathbb{P}\left(X_{\left\|\xi_{k}\right\|} \in \Delta_{\left\|\xi_{k}\right\|} \mid \xi_{k}\right) .
$$
In the same way, we can prove that for $\Delta_k\subset \mathbb{R}^{2k}$
$$
\lim_{t\to\infty}\frac{1}{t} \int_{0}^{t} \mu^{k}_{s}(\Delta_k) d s = \nu^k(\Delta_k) \quad \text { a.s. }
$$
Now, to finish the proof, we note that from 
$$
\lim_{t\to\infty}\frac{1}{t} \int_{0}^{t} \mu^k_{s}(\Delta_K) d s = \nu^k(\Delta_k) \quad \text { a.s. for } k\geq 1,
$$
it follows that 
$$
\lim_{t\to\infty}\frac{1}{t} \int_{0}^{t} G(\mu_{s}) d s = G(\nu) \quad \text { a.s.,}
$$
where $G$ is a functional on space of measures of the form
$$
G(\mu)=\sum_{k=1}^N\sum_{i=1}^n a_{k,i}\mu^k(\Delta_{k,i})
$$
for some constants $a_i$ and subsets $\Delta_i\subset \mathbb R^{2k}.$
\end{proof}
\end{example}

\begin{example}
As we mentioned in the Introduction, vortex evolution, as given by the Hamiltonian equation, can be considered a differential equation with interaction. We can write $N-$vortex system as
$$
\left\{\begin{array}{l}
d x(u, t)=-\frac{1}{\pi} \int_{\mathbb{R}^{2}} \varphi\left(x(u, t), y(v, t), u_{0}, v_{0}\right) \mu_{t}\left(d u_{0} d v_{0}\right) d t \\
d y(v, t)=\frac{1}{\pi} \int_{\mathbb{R}^{2}} \psi\left(x(u, t), y(v, t), u_{0}, v_{0}\right) \mu_{t}\left(d u_{0} d v_{0}\right) d t \\
x(u, 0)=u \\
y(v, v)=v \\
\mu_{t}=\mu_{0} \cdot(x(\cdot, t), y(\cdot, t))^{-1},
\end{array}\right.
$$
where $\varphi\left(u, v, u_{0}, v_{0}\right)=\frac{v-v_{0}}{\left(u-u_{0}\right)^{2}+\left(v-v_{0}\right)^{2}} \mathbb{I}_{\left\{v \neq v_{0}\right\}},$  $\psi\left(u, v, u_{0}, v_{0}\right)=\frac{u-u_{0}}{\left(u-u_{0}\right)^{2}+\left(v-v_{0}\right)^{2}} \mathbb{I}_{\left\{u{\neq u_{0}}\right\}},$
and $\mu_t$ is a shift forward of a discrete measure $\mu_0$
$$
\mu_{0}=\sum_{j=1}^{N} \frac{1}{\Gamma_{j}} \delta_{\left(u_{j}, v_{j}\right)} ; \quad \mu_{t}=\mu_{0} \cdot V(\cdot, t)^{-1}=\sum_{j=1}^{N} \frac{1}{\Gamma_{j}} \delta_{V\left(u_{j}, v_{j}; t\right)}.
$$

In general, for arbitrary initial position $\mu_0,$ for number of vortexes $N\geq 4$ there is no itegrable solution to such systems. There are works devoted to describtion of conditions under which trajectories form choreographies \cite{calleja2018}. For such cases, the trajectories have periodical strucutre and for any invariant of the braid, that is formed by them, has the limit of the form 
$$
\lim_{t\to \infty} \frac {m_k(t)}{t^k}= C_k.
$$

\end{example}


\

\bibliography{main}

\begin{thebibliography}{10}

\bibitem{Andrey-1-2003}
A.~A. Dorogovtsev.
\newblock Stochastic flows with interaction and measure-valued processes.
\newblock {\em International Journal of Mathematics and Mathematical Sciences},
  63(?):??, 2003.

\bibitem{Andrey-book}
A.~A. Dorogovtsev.
\newblock {\em Measure-valued processes and stochastic flows}.
\newblock Walter de Gruyter GmbH \& Co KG, 2023.

\bibitem{Arnold1978}
V.~I. Arnold.
\newblock {\em Mathematical Methods of Classical Mechanics}.
\newblock Springer, New York, 1978.

\bibitem{Moore1993}
C.~Moore.
\newblock Braids in classical dynamics.
\newblock {\em Physical Review Letters}, 70(24):3675--3679, 1993.

\bibitem{Clausen1998}
S.~Clausen, G.~Helgesen, and A.~T. Skjeltorp.
\newblock Braid description of collective fluctuations in a few-body system.
\newblock {\em Physical Review Letters}, 80(3):492--495, 1998.

\bibitem{Vassiliev1990}
V.~A. Vassiliev.
\newblock Cohomology of knot spaces.
\newblock {\em Theory of Singularities and Its Applications}, pages 23--69,
  1990.

\bibitem{Chmutov2012}
Sergei Chmutov, Sergei Duzhin, and Jacob Mostovoy.
\newblock {\em Introduction to Vassiliev Knot Invariants}.
\newblock Cambridge University Press, 2012.

\bibitem{BarNatan1995}
Dror Bar-Natan.
\newblock Vassiliev homotopy string link invariants.
\newblock {\em Journal of Knot Theory and Its Ramifications}, 4(1):13--32,
  1995.

\bibitem{Kontsevich1995}
Maxim Kontsevich.
\newblock Vassiliev's knot invariants.
\newblock {\em Advances in Soviet Mathematics}, 16:137--150, 1993.

\bibitem{Berger2001}
M.~A. Berger.
\newblock Topological invariants in braid theory.
\newblock {\em Reviews of Modern Physics}, 73(4):1143--1160, 2001.

\bibitem{Arnold-book}
V.I. Arnold.
\newblock {\em Mathematical methods of classical mechanics}.
\newblock Springer Science \& Business Media, 2013.

\bibitem{belashov2017interaction}
V.Yu. Belashov.
\newblock Interaction of n-vortex structures in a continuum, including
  atmosphere, hydrosphere and plasma.
\newblock {\em Advances in Space Research}, 60(8):1878--1890, 2017.

\bibitem{Woltjer1958}
L.~Woltjer.
\newblock A theorem on force-free magnetic fields.
\newblock {\em Proceedings of the National Academy of Sciences of the United
  States of America}, 44(6):489--491, 1958.

\bibitem{Moreau1961}
Jean-Jacques Moreau.
\newblock Constantes d'un \^{i}lot tourbillonnaire en fluide parfait barotrope.
\newblock {\em Comptes Rendus Hebdomadaires des S{\'e}ances de l'Acad{\'e}mie
  des Sciences}, 252:2810--2812, 1961.

\bibitem{Moffatt1969}
H.~K. Moffatt.
\newblock The degree of knottedness of tangled vortex lines.
\newblock {\em Journal of Fluid Mechanics}, 35(1):117--129, 1969.

\bibitem{Arnold1973}
V.~I. Arnold.
\newblock The asymptotic hopf invariant and its applications.
\newblock {\em Selecta Mathematica Sovietica}, 5(4):327--345, 1986.
\newblock Originally published in 1973 in *Trudy Moskov. Mat. Obshch.* 28,
  3--32.

\bibitem{ContrerasIturriaga1999}
Gonzalo Contreras and Renato Iturriaga.
\newblock Average linking numbers.
\newblock {\em Proceedings of the American Mathematical Society},
  127(5):1533--1539, 1999.

\bibitem{Vogel2003}
Thomas Vogel.
\newblock Asymptotic linking number and the helicity of volume-preserving
  vector fields.
\newblock {\em Journal of Physics A: Mathematical and General},
  36(29):7285--7294, 2003.

\bibitem{Kudryavtseva2016}
E.~A. Kudryavtseva.
\newblock Helicity is the only invariant of incompressible flows whose
  derivative is continuous in the $c^1$ topology.
\newblock {\em Journal of Mathematical Fluid Mechanics}, 18:407--416, 2016.

\bibitem{Enciso2016}
Alberto Enciso, Daniel Peralta-Salas, and Francisco~Torres de~Lizaur.
\newblock Helicity is the only integral invariant of volume-preserving
  transformations.
\newblock {\em Proceedings of the National Academy of Sciences},
  113(8):2035--2040, 2016.

\bibitem{BaaderMarche2012}
Sebastian Baader and Julien Marché.
\newblock Asymptotic vassiliev invariants for vector fields.
\newblock {\em Bulletin de la Société Mathématique de France},
  140(4):569--582, 2012.

\bibitem{SongBanksDiaz2006}
Yongmin Song, Stephen~P. Banks, and David Diaz.
\newblock Dynamical systems on three-manifolds. part i: Knots, links and chaos.
\newblock {\em International Journal of Bifurcation and Chaos},
  16(5):1381--1394, 2006.

\bibitem{Spitzer1958}
F.~Spitzer.
\newblock Some theorems concerning 2-dimensional brownian motion.
\newblock {\em Transactions of the American Mathematical Society}, 87:187--197,
  1958.

\bibitem{Shi1994}
Zhan Shi.
\newblock Asymptotic laws for the winding angles of planar brownian motion.
\newblock {\em Journal of Statistical Physics}, 73:433--440, 1993.

\bibitem{BertoinWerner1994}
Jean Bertoin and Wendelin Werner.
\newblock On the winding numbers of planar brownian motion.
\newblock {\em Annals of Probability}, 22(3):1128--1139, 1994.

\bibitem{Kuznetsov2015b}
V.~A. Kuznetsov.
\newblock Large deviations for the winding number of planar brownian motion.
\newblock {\em Theory of Probability and Its Applications}, 60(4):676--680,
  2015.

\bibitem{BergerRoberts1988}
M.~A. Berger and P.~H. Roberts.
\newblock On the winding number problem with finite steps.
\newblock {\em Journal of Physics A: Mathematical and General},
  21(5):999--1008, 1988.

\bibitem{Yor1991}
Marc Yor.
\newblock On some exponential functionals of brownian motion.
\newblock {\em Advances in Applied Probability}, 24(3):509--531, 1992.

\bibitem{ZirbelWoyczynski2002}
C.~L. Zirbel and W.~A. Woyczyński.
\newblock Mutual windings of planar brownian motions and stochastic flows.
\newblock {\em Journal of Theoretical Probability}, 15(1):1--25, 2002.

\bibitem{Kuznetsov2017}
V.~A. Kuznetsov.
\newblock On the winding numbers of isotropic brownian flows.
\newblock {\em Stochastic Processes and their Applications},
  127(10):3263--3280, 2017.

\bibitem{Dudley}
R.M. Dudley.
\newblock {\em Real analysis and probability}.
\newblock Chapman and Hall/CRC, 2018.

\bibitem{Ryogo-stoch.}
K.~Ryogo.
\newblock Stochastic liouville equations.
\newblock {\em Journal of Mathematical Physics}, 4(2):174--183, 1963.

\bibitem{Kunita}
H.~Kunita.
\newblock {\em Stochastic Flows and Stochastic Differential Equations}.
\newblock Cambridge University Press, 1997.

\bibitem{Halmos}
P.R. Halmos.
\newblock {\em Lectures on ergodic theory}.
\newblock Courier Dover Publications, 2017.

\bibitem{Kuznetsov2015}
V.A. Kuznetsov.
\newblock Kontsevich integral invariants for random trajectories.
\newblock {\em Ukrainian Mathematical Journal}, 67(1):62--73, 2015.

\bibitem{Ferrera}
J.Ferrera.
\newblock {\em An introduction to non-smooth analysis}.
\newblock ??, 2014.

\bibitem{CornfeldFominSinai}
Ya.G.~Sinai I.P.~Cornfeld, S.V.~Fomin.
\newblock {\em Ergodic theory}.
\newblock Springer, New York, 1981.

\bibitem{calleja2018}
Renato~C Calleja, Eusebuis~J. Doedel, and Carlos Carcia-Azpeita.
\newblock Choreographies in the n-vortex problem.
\newblock {\em Regular and Cahodic Dynamics}, 23(5):595--612, 2018.

\end{thebibliography}
\bibliographystyle{unsrt}
\end{document}